\documentclass[11pt,letterpaper]{article}%
\usepackage{amsfonts}
\usepackage{amssymb}
\usepackage{graphicx}
\usepackage{amsmath}
\usepackage{version}%
\providecommand{\U}[1]{\protect\rule{.1in}{.1in}}
\newtheorem{theorem}{Theorem}

\newtheorem{lemma}{Lemma}

\newtheorem{proposition}{Proposition}
\newtheorem{remark}{Remark}

\newenvironment{proof}[1][Proof]{\textbf{#1.} }{\ \rule{0.5em}{0.5em}}
\excludeversion{privatenotes}
\begin{document}

\title{Sharp Adaptive Nonparametric Testing for Sobolev Ellipsoids}
\author{\textsc{Pengsheng Ji}\textsc{\bigskip\ and Michael Nussbaum}\\\textit{University of Georgia and Cornell University }}
\date{}
\maketitle

\begin{abstract}
We consider testing for presence of a signal in Gaussian white noise with
intensity $n^{-1/2}$, when the alternatives are given by smoothness ellipsoids
with an $L_{2}$-ball of (squared) radius $\rho$ removed. It is known that, for
a fixed Sobolev type ellipsoid $\Sigma(\beta,M)$ of smoothness $\beta$ and
size $M$, a squared radius $\rho\asymp n^{-4\beta/(4\beta+1)}$ is the critical
separation rate, in the sense that the minimax error of second kind over
$\alpha$-tests stays asymptotically between $0$ and $1$ strictly (Ingster
\cite{Ingst-82}). In addition, Ermakov \cite{Ermak-90} found the sharp
asymptotics of the minimax error of second kind at the separation rate. For
adaptation over both $\beta$ and $M$ in that context, it is known that a
$\log\log$-penalty over the separation rate for $\rho$ is necessary for a
nonzero asymptotic power. Here, following an example in nonparametric
estimation related to the Pinsker constant, we investigate the adaptation
problem over the ellipsoid size $M$ only, for fixed smoothness degree $\beta$.
It is established that the sharp risk asymptotics can be replicated in that
adaptive setting, if $\rho\rightarrow0$ more slowly than the separation rate.
The penalty for adaptation here turns out to be a sequence tending to infinity
arbitrarily slowly.

\end{abstract}

\begin{table}[b]
\rule{3cm}{0.01cm} \newline
\par
{\small Supported in part by NSF\ Grant DMS-08-05632
\newline}
{\small 2000 Mathematics Subject Classification: 62G10, 62G20 } \newline%
{\small \textit{Key words and phrases:} minimax hypothesis testing,
nonparametric signal detection, sharp asymptotic adaptivity, moderate
deviation }\end{table}

\section{Introduction and main result\label{sec:intro-and-main}}

Consider the Gaussian white noise model in sequence space, where observations
are
\begin{equation}
Y_{j}=f_{j}+n^{-1/2}\xi_{j},\qquad j=1,2,..., \label{gwn-1}%
\end{equation}
with unknown, nonrandom signal $f=(f_{j})_{j=1}^{\infty}$, and noise variables
$\xi_{j}$ which are i.i.d. $N(0,1)$. We intend to test the null hypothesis of
\textquotedblleft no signal\textquotedblright\ against nonparametric
alternatives described as follows. For some $\beta>0$ and $M>0$, let
$\Sigma(\beta,M)$ be the set of sequences
\[
\Sigma(\beta,M)=\{f=(f_{j})_{j=1}^{\infty}:\sum_{j=1}^{\infty}j^{2\beta}%
f_{j}^{2}\leq M\};
\]
this might be called a Sobolev type ellipsoid with smoothness parameter
$\beta$ and size parameter $M$. Consider further the complement of an open
ball in the sequence space $l_{2}$: if $\left\Vert f\right\Vert _{2}^{2}%
=\sum_{j=1}^{\infty}f_{j}^{2}$ is the squared norm then
\[
B_{\rho}=\{f\in l_{2}:\left\Vert f\right\Vert _{2}^{2}\geq\rho\}.
\]
Here $\rho^{1/2}$ is the radius of the open ball; by an abuse of language we
call $\rho$ itself the \textquotedblleft radius\textquotedblright. We study
the hypothesis testing problem
\[
H_{0}:f=0\quad\text{ against }\quad H_{a}:f\in\Sigma(\beta,M)\cap B_{\rho}.
\]
Assuming that $n\rightarrow\infty$, implying that the noise size $n^{-1/2}$
tends to zero, we expect that for a fixed radius $\rho$, consistent $\alpha
$-testing in that setting is possible. More precisely, there exist $\alpha
$-tests with type II error tending to zero uniformly over the nonparametric
alternative $f\in\Sigma(\beta,M)\cap B_{\rho}$. If now the radius $\rho
=\rho_{n}$ tends to zero as $n\rightarrow\infty$, the problem becomes more
difficult and if $\rho_{n}\rightarrow0$ too quickly, all $\alpha$-tests will
have the trivial asymptotic (worst case) power $\alpha$. According to a
fundamental result of Ingster \cite{Ingst-82} there is a critical rate for
$\rho_{n}$, the so-called \textit{separation rate}%
\begin{equation}
\rho_{n}\asymp n^{-4\beta/(4\beta+1)} \label{separ-rate}%
\end{equation}
at which the transition in the power behaviour occurs. More precisely,
consider a (possibly randomized) $\alpha$-test $\phi_{n}$ in the model
(\ref{gwn-1}) for null hypothesis $H_{0}:f=0$, that is, a test fulfilling
$E_{n,0}\phi_{n}\leq\alpha$ where $E_{n,f}\left(  \cdot\right)  $ denotes
expectation in the model (\ref{gwn-1}).
For given $\phi_{n}$, we define the worst case type II error over the
alternative $f\in\Sigma(\beta,M)\cap B_{\rho}$ as
\begin{equation}
\Psi(\phi_{n},\rho,\beta,M):=\sup_{f\in\Sigma(\beta,M)\cap B_{\rho}}\left(
1-E_{n,f}\phi_{n}\right)  . \label{worst-case-type-II-error-def}%
\end{equation}
The search for a best $\alpha$-test in this sense leads to the minimax type II
error
\begin{equation}
\pi_{n}(\alpha,\rho,\beta,M):=\inf_{\phi_{n}:E_{n,0}\phi_{n}\leq\alpha}%
\Psi(\phi_{n},\rho,\beta,M). \label{minimax-type-II-error-def}%
\end{equation}
An $\alpha$-test which attains the infimum above for a given $n$ is minimax
with respect to type II error. Ingster's separation rate result can now be
formulated as follows: if $\rho_{n}\asymp n^{-4\beta/(4\beta+1)}$ and
$0<\alpha<1$ then
\[
0<\liminf_{n}\pi_{n}(\alpha,\rho_{n},\beta,M)\text{ and }\limsup_{n}\pi
_{n}(\alpha,\rho_{n},\beta,M)<1-\alpha.
\]
Moreover, if $\rho_{n}\gg n^{-4\beta/(4\beta+1)}$ then $\pi_{n}(\alpha
,\rho_{n},\beta,M)\rightarrow0$, and if $\rho_{n}\ll n^{-4\beta/(4\beta+1)}$
then $\pi_{n}(\alpha,\rho_{n},\beta,M)\rightarrow1-\alpha$.

These minimax rates in nonparametric testing, presented here in the simplest
case of an $l_{2}$-setting, have been extended in two ways. In the first of
these, Ermakov \cite{Ermak-90} found the exact asymptotics of the minimax type
II error $\pi_{n}(\alpha,\rho,\beta,M)$ (equivalently, of the maximin power)
at the separation rate. The shape of that result and its derivation from an
underlying Bayes-minimax theorem on ellipsoids exhibit an analogy to the
Pinsker constant in nonparametric estimation. In another direction, Spokoiny
\cite{spokoi-96} considered the adaptive version of the minimax nonparametric
testing problem, where both $\beta$ and $M$ are unknown, and showed that the
rate at which $\rho_{n}\rightarrow0$ has to be slowed down by a $\log\log
n$-factor if nontrivial asymptotic power is to be achieved. Thus an
\textquotedblleft adaptive minimax rate\textquotedblright\ was specified,
analogous to Ingster's nonadaptive separation rate (\ref{separ-rate}), where
the additional $\log\log n$-factor is interpreted as a penalty for adaptation.
However this result did not involve a sharp asymptotics of type II error in
the sense of \cite{Ermak-90}.

It is noteworthy that in nonparametric estimation over $f\in\Sigma(\beta,M)$
with $l_{2}$-loss (as opposed to testing), where the risk asymptotics is given
by the Pinsker constant, there is a multitude of results showing that
adaptation is possible with neither a penalty in the rate nor in the constant,
cf. Efromovich and Pinsker \cite{Efrom-Pinsk-learning}, Golubev
\cite{Golub-87}, \cite{gol-90-quasi}, Tsybakov \cite{Tsyb-book}. The present
paper deals with the question of whether the sharp risk asymptotics for
testing in the sense of \cite{Ermak-90} can be reproduced in an adaptive
setting, in the context of a possible rate penalty for adaptation.

Let us present the well known result on sharp risk asymptotics for testing in
the nonadaptive setting. Let $\Phi$ be the distribution function of the
standard normal, and for $\alpha\in\left(  0,1\right)  $ let $z_{\alpha}$ be
the upper $\alpha$-quantile, such that $\Phi(z_{\alpha})=1-\alpha$. Write
$a_{n}\gg b_{n}$ (or $b_{n}\ll a_{n}$) iff $b_{n}=o(a_{n})$, and $a_{n}\sim
b_{n}$ iff $\lim_{n}a_{n}/b_{n}=1$.

\begin{proposition}
\label{propos-ermak-90}(Ermakov \cite{Ermak-90}) Suppose $\alpha\in(0,1)$ and
that the radius $\rho_{n}$ tends to zero at the separation rate, more
precisely
\[
\rho_{n}\sim c\cdot n^{-4\beta/(4\beta+1)}%
\]
for some constant $c>0$. \newline(i) For any sequence of tests $\phi_{n}$
satisfying $E_{n,0}\phi_{n}\leq\alpha+o(1)$ we have
\[
\Psi(\phi_{n},\rho_{n},\beta,M)\geq\Phi(z_{\alpha}-\sqrt{A(c,\beta
,M)/2})+o(1)\text{ as }n\rightarrow\infty,
\]
where%
\begin{equation}
A(c,\beta,M)=A_{0}(\beta)M^{-1/(2\beta)}c^{2+1/2\beta} \label{A-c-beta-M-def}%
\end{equation}
and $A_{0}(\beta)$ is Ermakov's constant
\begin{equation}
A_{0}(\beta)=\frac{2(2\beta+1)}{(4\beta+1)^{1+1/2\beta}}. \label{A-0-beta-def}%
\end{equation}
\newline(ii) For every $M>0$ there exists a sequence of tests $\phi_{n}$
satisfying $E_{n,0}\phi_{n}\leq\alpha+o(1)$ such that
\[
\Psi(\phi_{n},\rho_{n},\beta,M)\leq\Phi(z_{\alpha}-\sqrt{A(c,\beta
,M)/2})+o(1).
\]

\end{proposition}

This gives the sharp asymptotics for the minimax type II error at the
separation rate, analogous to the Pinsker constant \cite{pinsk-80} for
nonparametric estimation. The optimal test attaining the bound of (ii) above,
as given in \cite{Ermak-90}, depends on $\beta$ and $M$. Concerning adaptivity
in both of these parameters, the following result is known.\pagebreak

\begin{proposition}
\label{propos-Spok-96}(Spokoiny \cite{spokoi-96}). Let $\mathcal{T}$ be a
subset of $(0,\infty)\times(0,\infty)$ such that there exist $M>0$, $\beta
_{2}>\beta_{1}>0$ and
\[
\mathcal{T}\supseteq\{(\beta,M):\beta_{1}\leq\beta\leq\beta_{2}\}.
\]
(i) If $t_{n}\ll(\log\log n)^{1/2}$ and $\rho_{n}\sim c\cdot(n/t_{n}%
)^{-4\beta/(4\beta+1)}$, then for any $c>0$ and any sequence of tests
$\phi_{n}$ satisfying $E_{n,0}\phi_{n}\leq\alpha+o(1)$, and not depending on
$\beta$or $M$, we have
\[
\sup_{(\beta,M)\in\mathcal{T}}\Psi(\phi_{n},\rho_{n},\beta,M)\geq
1-\alpha+o(1).
\]
\newline(ii) For any $\beta^{\ast}>1/2$ and $0<M_{1}\leq M_{2}$, let
\[
\mathcal{T}=\{(\beta,M):1/2<\beta\leq\beta^{\ast},M_{1}\leq M\leq M_{2}\}.
\]
Then there exist a constant $c_{1}=c_{1}(\beta^{\ast},M_{1},M_{2})$ and a
sequence of tests $\phi_{n}$ satisfying $E_{n,0}\phi_{n}=o(1)$ such that, if
\begin{equation}
\rho_{n}\sim c_{1}\left(  \frac{n}{(\log\log n)^{1/2}}\right)  ^{-4\beta
/(4\beta+1)} \label{adapt-minimax-rate}%
\end{equation}
then
\begin{equation}
\sup_{(\beta,M)\in\mathcal{T}}\Psi(\phi_{n},\rho_{n},\beta,M)=o(1).
\label{worst-case-adaptive-error}%
\end{equation}

\end{proposition}

Here the criterion to evaluate a test sequence has changed, to include the
worst case type II error over a whole range of $\beta,M$. Hence the critical
radius rate (\ref{adapt-minimax-rate}) has to be interpreted as an
\textit{adaptive separation rate. }It differs by a factor $(\log\log
n)^{2\beta/(4\beta+1)}$ from the nonadaptive separation rate (\ref{separ-rate}%
); this factor is an example of the well-known phenomenon of a penalty for
adaptation. Furthermore, as noted in \cite{spokoi-96}, a degenerate behaviour
occurs here, in that both error probabilities at the critical rate tend to
zero. Thus any sequence $\phi_{n}$ of tests fulfilling
(\ref{worst-case-adaptive-error}) should be seen as \textit{adaptive rate
optimal}, comparable to rate optimal tests in the nonadaptive case (that is,
tests fulfilling $\limsup_{n}\Psi(\phi_{n},\rho_{n},\beta,M)<1-\alpha$ at
$\rho_{n}$ given by (\ref{separ-rate})). In Ingster and Suslina
\cite{Ingst-Sus-book}, chap. 7, the worst case adaptive error
(\ref{worst-case-adaptive-error}) is further analyzed, with a view to a sharp
asymptotics; cf. Remark \ref{rem:Ingster-adaptivity} below for a discussion in
relation to our results.

In this paper we address the question of whether an exact type II error
asymptotics in the sense of \cite{Ermak-90} is possible in an adaptive
setting. In our approach $\beta$ is kept fixed, while we aim for adaptation
over the ellipsoid size $M$. First, we present a negative result for
adaptation at the classical separation rate (\ref{separ-rate}).

\begin{theorem}
Suppose $c>0$, $0<M_{1}<M_{2}<\infty$ and $\rho_{n}\sim c\cdot n^{-4\beta
/(4\beta+1)}$. Then there is no test $\phi_{n}$ satisfying $E_{n,0}\phi
_{n}\leq\alpha+o(1)$, not depending on $i=1,2$ but satisfying both relations
\[
\Psi_{n}(\phi_{n},\rho_{n},\beta,M_{i})\leq\Phi(z_{\alpha}-\sqrt
{A(c,\beta,M_{i})/2})+o(1),\;\;\;i=1,2.
\]
\label{adaptationimpossible}
\end{theorem}

This result states that adaptation even just over $M$ is impossible at the
separation rate. Instead, we enlarge the radius slightly and examine how the
minimax error approaches zero. To be specific, we replace the constant $c$ in
$\rho_{n}\sim c\cdot n^{-4\beta/(4\beta+1)}$ by a sequence $c_{n}$ tending to
infinity slowly. In that case the minimax type II\ error bound of Proposition
\ref{propos-ermak-90}, namely $\Phi(z_{\alpha}-\sqrt{A(c,\beta,M)/2})$ will
tend to zero (since $A(c,\beta,M)$ as defined in (\ref{A-c-beta-M-def})
contains a factor $c^{2+1/(2\beta)}$). When the log-asymptotics of this error
probability is considered, as in moderate and large deviation theory, it turns
out that adaptation to Ermakov%
\'{}%
s constant is possible.

\begin{theorem}
Assume $c_{n}\rightarrow\infty$ and $c_{n}=o(n^{K})$ for every $K>0$. If
$\rho_{n}=c_{n}\cdot n^{-4\beta/(4\beta+1)}$ then there exists a test
$\phi_{n}$ not depending on $M$ such that
\[
E_{n,0}\phi_{n}\leq\alpha+o(1),
\]
and for all $M>0$
\[
\limsup_{n}\frac{1}{c_{n}^{2+1/(2\beta)}}\log\Psi(\phi_{n},\rho_{n}%
,\beta,M)\leq-\frac{A_{0}(\beta)M^{-1/(2\beta)}}{4}.
\]
\label{sharpadaptation}
\end{theorem}

However now, since the optimality criterion has been changed, a formal
argument is needed that no $\alpha$-test can be better in the sense of the
log-asymptotics for the error of second kind. Such a result is implied by
Theorem 3 in Ermakov \cite{Ermak-08}, where the nonadaptive sharp asymptotics
is studied in a setting where $\rho_{n}=c_{n}\cdot n^{-4\beta/(4\beta+1)}$
with $c_{n}\rightarrow\infty$, hence type II error probability tends to zero.

\begin{proposition}
Under the assumptions of the previous theorem, any test $\phi_{n}$ (possibly
depending on $M$) satisfying $E_{n,0}\phi_{n}\leq\alpha+o(1)$ also fulfills
\begin{equation}
\liminf_{n}\frac{1}{c_{n}^{2+1/(2\beta)}}\log\Psi(\phi_{n},\rho_{n}%
,\beta,M)\geq-\frac{A_{0}(\beta)M^{-1/(2\beta)}}{4}.
\label{lower-bound-mod-dev-testing}%
\end{equation}

\end{proposition}

This result is implied by Theorem 3 in \cite{Ermak-08}, and hence the proof is omitted.

To further discuss the context of the main results, we note the following points.

\begin{remark}
\textit{ Logarithmic vs. strong asymptotics}\textbf{. }\emph{In
\cite{Ermak-08} it is also shown that, for nonadaptive testing where }%
$\rho_{n}=c_{n}\cdot n^{-4\beta/(4\beta+1)}$\emph{, }$c_{n}\rightarrow\infty
$\emph{, the lower bound (\ref{lower-bound-mod-dev-testing}) is attainable, so
that the minimax type II error defined by (\ref{minimax-type-II-error-def})
satisfies }%
\begin{equation}
\log\pi_{n}(\alpha,\rho_{n},\beta,M)\sim-\frac{1}{4}A(c_{n},\beta,M).
\label{log-asymptotics-minimax-err}%
\end{equation}
\emph{This holds as long as }$\rho_{n}\ll n^{-2\beta/(2\beta+1)}$\emph{.
Moreover if additionally }$\rho_{n}\ll n^{-3\beta/(3\beta+1)}$\emph{ then the
log-asymptotics (\ref{log-asymptotics-minimax-err}) can be strengthened to }%
\begin{equation}
\pi_{n}(\alpha,\rho_{n},\beta,M)\sim\Phi(z_{\alpha}-\sqrt{A(c_{n},\beta
,M)/2}). \label{strong-asymptotics-minimax-err}%
\end{equation}
\emph{Results (\ref{log-asymptotics-minimax-err}) and
(\ref{strong-asymptotics-minimax-err}) have been obtained within a framework
of efficient inference for moderate deviation probabilities, cf. Ermakov
\cite{Ermak-MD}, \cite{Ermak-11}. Recall that in our setting }$c_{n}=o(n^{K}%
)$\emph{ for every }$K>0$\emph{, so that the strong asymptotics
(\ref{strong-asymptotics-minimax-err}) holds in the nonadaptive setting. It is
an open question whether an adaptive analog of
(\ref{strong-asymptotics-minimax-err}) holds.}\newline\emph{For standardized
sums }$T_{n}$\emph{ of independent random variables, if }$\{T_{n}>x_{n}%
\}$\emph{ is a large or moderate deviation event, theorems on the relative
error caused by replacing the exact distribution of }$T_{n}$\emph{ by its
limiting distribution are sometimes called strong large or moderate deviation
theorems to distinguish them from first order results on }$\log P(T_{n}%
>x_{n})$\emph{. For a background cf. \cite{Petrov-book},
\cite{Inglot-et-al-92}, \cite{Chen-et-al-book}, chap. 11.}
\end{remark}

\begin{remark}
\label{rem:Ingster-adaptivity}\textit{Sharp asymptotics with both }$\beta
,M$\textit{ unknown}. \emph{The adaptivity result of Spokoiny \cite{spokoi-96}%
, discussed in Proposition \ref{propos-Spok-96}, about the rate penalty for
adaptation }$(\log\log n)^{2\beta/(4\beta+1)}$\emph{, does not provide a sharp
risk asymptotics in the sense of either Proposition \ref{propos-ermak-90} or
our Theorems \ref{adaptationimpossible} and \ref{sharpadaptation}. Some
results in this direction are presented in section 7.1.3 of Ingster and
Suslina \cite{Ingst-Sus-book}. To clarify the relation to our setting where
}$\beta$\emph{ is fixed and adaptivity refers to the size parameter }%
$M$\emph{, let us discuss these results here.}\newline\emph{Let us first
reformulate the result of Proposition \ref{propos-ermak-90} (that is
\cite{Ermak-90}) for known }$\beta,M$\emph{ in a certain dual way, where a
given type II error is prescribed and it is shown to be attainable on a radius
sequence }$\rho_{n}$\emph{ which then varies with }$\beta,M$\emph{. Suppose
}$\alpha\in(0,1)$\emph{ and }$d>0$\emph{ are given, and suppose the radius
}$\rho_{n}$\emph{ satisfies}%
\[
\rho_{n}^{(4\beta+1)/4\beta}\sim n^{-1}A_{1}\left(  \beta\right)  M^{1/4\beta
}d
\]
\emph{where }$A_{1}\left(  \beta\right)  =\left(  A_{0}\left(  \beta\right)
/2\right)  ^{-1/2}$\emph{, and }$A_{0}\left(  \beta\right)  $\emph{ is given
by (\ref{A-0-beta-def}). Then for any sequence of tests }$\phi_{n}$\emph{
satisfying }$E_{n,0}\phi_{n}\leq\alpha+o(1)$\emph{ we have }%
\[
\Psi(\phi_{n},\rho_{n},\beta,M)\geq\Phi(z_{\alpha}-d)+o(1)\text{ as
}n\rightarrow\infty,
\]
\emph{and there is a sequence }$\phi_{n}$\emph{ (depending on }$\beta
,M$\emph{) attaining this lower bound. This follows directly from Proposition
\ref{propos-ermak-90} by setting }$d=\sqrt{A(c,\beta,M)/2}$\emph{ and solving
for }$c$\emph{.}\newline\emph{In the setting of \cite{Ingst-Sus-book}, the
smoothness parameter }$\beta$\emph{ varies over a range }$\left[  \beta
_{1},\beta_{2}\right]  $\emph{, as in Proposition \ref{propos-Spok-96}. To
state the lower asymptotic risk bound, assume that }$0<\beta_{1}<\beta_{2}%
$\emph{, that }$M>0$\emph{ is fixed and define }%
\[
\mathcal{T}=\{(\beta,M):\beta_{1}\leq\beta\leq\beta_{2}\}.
\]
\emph{Let }$D\in R$\emph{ be arbitrary and define a radius sequence }%
$\rho_{n,\beta,M}$\emph{ by }%
\begin{equation}
\left(  \rho_{n,\beta,M}\right)  ^{(4\beta+1)/4\beta}=n^{-1}A_{1}\left(
\beta\right)  M^{1/4\beta}\left(  \left(  2\log\log n\right)  ^{1/2}+D\right)
. \label{sharp-adapt-radius-def}%
\end{equation}
\emph{The lower asymptotic risk bound (a variation of Theorem 7.1 in
\cite{Ingst-Sus-book}) can then be formulated as follows. For any sequence of
tests }$\phi_{n}$\emph{ satisfying }$E_{n,0}\phi_{n}\leq\alpha+o(1)$\emph{ we
have }%
\begin{equation}
\sup_{(\beta,M)\in\mathcal{T}}\Psi(\phi_{n},\rho_{n,\beta,M},\beta
,M)\geq\left(  1-\alpha\right)  \Phi\left(  -D\right)  +o(1).
\label{ingster-adaptive-lower-bound}%
\end{equation}
\emph{Note in this setting, the test sequences }$\phi_{n}$\emph{ are assumed
not to depend on }$\beta$\emph{ but the radius }$\rho_{n,\beta,M}$\emph{ does.
Note that part (i) of Proposition \ref{propos-Spok-96} is implied by
(\ref{ingster-adaptive-lower-bound}) by letting }$D\rightarrow-\infty$%
\emph{.}\newline\emph{As to the attainability of this bound, the test provided
in section 7.3 of \cite{Ingst-Sus-book} depends on }$M$\emph{. Indeed in
\cite{Ingst-Sus-book} observations are assumed to be }$X_{j}=v_{j}+\xi_{j}%
$\emph{, where }$\xi_{j}$\emph{ are i.i.d. standard normal and }$v=\left(
v_{j}\right)  _{j=1}^{\infty}$\emph{ satisfies restrictions }$\sum_{j}%
v_{j}^{2}\geq r^{2}$\emph{, }$\sum_{j}j^{2\beta}v_{j}^{2}\leq R^{2}$\emph{
where }$R\rightarrow\infty$\emph{ and }$r/R\rightarrow0$\emph{ (the "power
norm" case in the book, where }$p=q=2,s=\beta;$\emph{ also }$r$\emph{ is
}$\rho$\emph{ in \cite{Ingst-Sus-book}). This observation model is equivalent
to ours upon setting }$R^{2}=nM$\emph{, }$r^{2}=n\rho$\emph{, and then }%
$Y_{j}=n^{-1/2}X_{j}$\emph{, }$f_{j}=n^{-1/2}v_{j}$\emph{. The reasoning
provided in section 7.3.2 of \cite{Ingst-Sus-book} makes it clear that the
test constructed uses solutions of an extremal problem under restrictions
}$\left\{  v:\sum_{j}v_{j}^{2}\geq r^{2},\sum_{j}j^{2\beta}v_{j}^{2}\leq
R^{2}\right\}  $\emph{ where }$r^{2}=n\rho_{n,\beta,M}$\emph{ with }%
$\rho_{n,\beta,M}$\emph{ from (\ref{sharp-adapt-radius-def}) and }$\beta
$\emph{ is from a certain grid of values in }$(\beta_{1},\beta_{2})$\emph{.
Since in particular }$R=n^{1/2}M^{1/2}$\emph{, it turns out that the estimator
depends on }$M$\emph{, though it has been made independent of }$\beta\in
(\beta_{1},\beta_{2})$\emph{. A version of such results for }$\alpha_{n}%
$\emph{-tests with }$\alpha_{n}\rightarrow0$\emph{ is given in
\cite{Ingst-Sus-05}.\newline It should be noted that adaptation to }$\beta
$\emph{ only, with }$M$\emph{ remaining fixed, does not have a practical
interpretation in the context of smooth functions. Thus the problem of a sharp
risk bound for adaptation to }$\left(  \beta,M\right)  $\emph{ remains open in
nonparametric testing; for the analogous problem in the estimation case
(regarding the Pinsker bound), solutions have been presented by Golubev
\cite{gol-90-quasi} and Tsybakov \cite{Tsyb-book}, sec 3.7.}
\end{remark}

\begin{remark}
\textit{The detection problem.}\emph{ Instead of focussing on the worst case
type II\ error }$\Psi(\phi_{n},\rho,\beta,M)$\emph{
(\ref{worst-case-type-II-error-def})\ of }$\alpha$\emph{-tests }$\phi_{n}%
$\emph{, one may consider minimization of the sum of errors, that is of
}$E_{n,0}\phi_{n}+\Psi(\phi_{n},\rho,\beta,M)$\emph{, over all tests }%
$\phi_{n}$\emph{. That has been called the detection problem in the
literature; in \cite{Ingst-Sus-book} this problem is largely treated in
parallel to the one for }$\alpha$\emph{-tests. There and in
\cite{Ingst-Sus-11} one finds the analog of the nonadaptive sharp asymptotics
of Proposition \ref{propos-ermak-90}. It may be conjectured that analogs of
our Theorems \ref{adaptationimpossible} and \ref{sharpadaptation} concerning
adaptivity hold there as well.}
\end{remark}

\begin{remark}
\textbf{ } \textit{The plug-in method}. \emph{In the present setting, where
the degree of smoothness }$\beta$\emph{ is fixed but the ellipsoid size }%
$M$\emph{ is unknown, a natural approach to adaptivity is to try to estimate
}$M$\emph{ and use a plug-in method. However uniformly consistent estimators
of }$M$\emph{ do not exist (since the unit ball in }$L_{2}$\emph{ is not
compact), hence for minimax optimality, such a straighforward argument fails.
In the estimation setting, the solution found by Golubev \cite{Golub-87} is to
apply, for a biased estimator of }$M$\emph{, the same saddle point reasoning
which lies at the heart of the Pinsker \cite{pinsk-80} result about minimax
optimal estimation. The paper \cite{Golub-87} concerns the continuous white
noise model indexed by }$t\in\left[  0,1\right]  $\emph{, and the adaptivity
there incorporates two local aspects: one with respect to time }$t\in\left[
0,1\right]  $\emph{ and the other with respect to a local variant of Sobolev
smoothness classes. For more discussion cf. \cite{Gol-Nus-90}.}\newline%
\emph{Our result here is the analog of the one by Golubev \cite{Golub-87} for
estimation, but in testing it turns out that adaptivity is possible only in
conjunction with a tail probability (moderate deviation) approach. To further
clarify the connection to adaptive estimation, in section
\ref{sec:adaptive-estim} we present a short outline of the result of
\cite{Golub-87} in a simplified setting.}
\end{remark}

\begin{remark}
\textit{Quadratic functionals.}\textbf{ }\emph{In the literature it has been
noted that the nonparametric testing problem with an }$l_{2}$\emph{-ball
removed is related to the estimation problem of the quadratic functional
}$Q(f)=\left\Vert f\right\Vert _{2}^{2}$\emph{. In particular, it is known
that the optimal separation rate for testing }$\rho_{n}^{1/2}\asymp
n^{-2\beta/(4\beta+1)}$\emph{ (comp. (\ref{separ-rate})) and the minimax
optimal rate for estimating }$Q(f)$\emph{ over }$\Sigma(\beta,M)$\emph{
coincide if }$0<\beta<1/4$\emph{, but if }$\beta\geq1/4$\emph{ then the latter
rate becomes }$n^{-1/2}$\emph{ (the so-called elbow effect; cf. Klemel\"{a}
\cite{klemela} and references therein). Butucea \cite{Butucea-07} gave a
unified argument for lower bounds in the estimation and testing cases when
rates coincide. As far as adaptive estimation rates for }$Q(f)$\emph{ are
concerned, the logarithmic penalty factor in the "irregular" case }%
$0<\beta<1/4$\emph{ has been established in \cite{Efrom-Low-quadratic}. In
\cite{Efrom-94} it has been shown that at the point }$\beta=1/4$\emph{ the
optimal adaptive rate is }$n^{-1/2}c_{n}$\emph{ where }$c_{n}\rightarrow
\infty$\emph{ slower than any power function of }$n$\emph{, and for }%
$\beta>1/4$\emph{, there is no adaptation penalty on the optimal rate
}$n^{-1/2}$\emph{. In the case }$0<\beta<1/4$\emph{, the only sharp adaptive
minimaxity result for estimation of }$Q(f)$\emph{ we are aware of is in
\cite{klemela}; it concerns a case where the }$l_{2}$\emph{-Sobolev class
}$\Sigma(\beta,M)$\emph{ is replaced by an }$l_{p}$\emph{-smoothness body with
}$p=4$\emph{.}
\end{remark}

\begin{remark}
\textit{ The sup-norm problem.}\textbf{ }\emph{Lepski and Tsybakov
\cite{Lep-Tsybak} proved a sharp minimax result in testing when the
alternative is a H\"{o}lder class (denoted }$H\left(  \beta,L\right)  $\emph{,
say) with an sup-norm ball removed, which is a testing analog of the minimax
estimation result of Korostelev \cite{Korost} and also a sup-norm analog of
Ermakov \cite{Ermak-90}. For adaptive minimax estimation with unknown }%
$(\beta,L)$\emph{ in the sup-norm case cf. \cite{Gol-Lep}; for the testing
case where }$\beta$\emph{ is given, D\"{u}mbgen and Spokoiny
\cite{Duembgen-Spok} established a sharp adaptivity result with respect to the
size parameter }$L$\emph{ only. The result in Theorem 2.2. of
\cite{Duembgen-Spok} can be seen as a analog of the one given here, although
the methodology in the sup-norm case is much different due to the connection
to deterministic optimal recovery, cf. \cite{Lep-Tsybak}. The case of unknown
}$(\beta,L)$\emph{ seems to be an open problem in the sup-norm testing case,
with regard to sharp minimaxity, although in \cite{Duembgen-Spok} a test is
given which is adaptive rate optimal without a }$\log\log n$\emph{-type
penalty. Rohde \cite{rohde} discusses the sup-norm case for regression with
nongaussian errors, combining methods of \cite{Duembgen-Spok} with ideas
related to rank tests.}
\end{remark}

\begin{remark}
\textit{Density, regression and other models. } \emph{The phenomenon of the
}$\log\log n$\emph{-type penalty in the rate for adaptation when an }$L_{2}%
$\emph{-ball is removed, as found by \cite{spokoi-96}, has also been
established in a discrete regression model \cite{Gayraud-Pouet}, and in
density models with direct and indirect observations \cite{Fromont-Laurent},
\cite{Butu-Mat-Pouet}. For a review of adaptive separation rates and further
results in a Poisson process model cf. \cite{From-Laurent-et-al-Poisson}.}
\end{remark}

The structure of the paper is as follows. In Section \ref{sec:Bayes-minimax},
we discuss the background, for the nonadaptive setting, of the sharp
asymptotic minimaxity result for testing of Ermakov \cite{Ermak-90} and its
analogy to the Pinsker \cite{pinsk-80} constant. In Section
\ref{sec:adaptation-impossible} we present the proof of Theorem
\ref{adaptationimpossible} about the lower bound (the necessary penalty) for
adaptation and in Section \ref{sec:sharp-adaptation}, Theorem
\ref{sharpadaptation} concerning attainability is proved. In an appendix
(Section \ref{sec:adaptive-estim}), we present some more background for the
reader, by giving a brief sketch of the estimation analog of our nonparametric
testing result (Golubev \cite{Golub-87}). Finally, Section
\ref{sec: proofs-Bayes-minimax} contains some proofs for the background
Section \ref{sec:Bayes-minimax}.

\section{The Bayes-minimax problem for nonparametric
testing\label{sec:Bayes-minimax}}

The purpose of this expository section is to elucidate the analogy between the
Pinsker constant \cite{pinsk-80} for $l_{2}$-estimation over ellipsoids and
the constant found by Ermakov \cite{Ermak-90} for nonparametric testing over
ellipsoids with an $l_{2}$-ball removed. We draw on the backgound explanation
given in \cite{Ingst-Sus-book}, sec. 4.1, but we focus specifically on the
fact that very similar Bayes-minimax problems are at the root of the
estimation and testing variants. For the theory underlying the Pinsker
constant cf. \cite{Bel-Lev}, \cite{Nu-survey}, \cite{Tsyb-book}.

For this exposition, we shall assume that observations (\ref{gwn-1}) are for
$j=1,\ldots,n$; we will thus assume $f\in\mathbb{R}^{n}$ and understand the
sets $\Sigma(\beta,M)$ and $B_{\rho}$ accordingly, i.e. they refer only to the
first $n$ coefficents of $f$. By $\left\Vert \cdot\right\Vert $ and
$\left\langle \cdot,\cdot\right\rangle $ we denote euclidean norm and inner
product in $\mathbb{R}^{n}$. Since most expressions will depend on $n$, for
this discussion we shall often suppress dependence on $n$ in the notation.
Assume that the radius $\rho$ tends to zero at the critical rate, that is
$\rho\asymp n^{-4\beta/(4\beta+1)}$. Let $\mathbb{R}_{+}^{n}=\left[
0,\infty\right)  ^{n}$; for a certain $d\in\mathbb{R}_{+}^{n}$, consider a
quadratic statistic of the form $\tilde{T}=n\sum_{j=1}^{n}d_{j}Y_{j}^{2}$.
Under $H_{0}$, we have $E_{0,n}\tilde{T}=\sum_{j=1}^{n}d_{j}$ and
$\mathrm{Var}_{0,n}\tilde{T}=2\left\Vert d\right\Vert ^{2}$. Since we will
work with the normalized test statistic, obtained by centering and dividing by
the standard deviation, it is obvious that we need only consider coefficients
$d$ fulfilling $\left\Vert d\right\Vert ^{2}=1$. Accordingly define, for such
coefficients $d$, the statistic
\begin{equation}
T=\frac{1}{\sqrt{2}}\left(  \tilde{T}-\sum_{j=1}^{n}d_{j}\right)  .
\label{Tn-tilde-def}%
\end{equation}
Under $H_{0}$, we now have $E_{0}T=0$ and $\mathrm{Var}_{0}T=1.$ We will
consider quadratic tests%
\begin{equation}
\psi_{d}=\mathbf{1}\left\{  T>z_{\alpha}\right\}  . \label{quad-tests-def}%
\end{equation}
A further condition on $d$ is imposed by requiring $d\in\mathcal{D}$, a set
which is defined for a given sequence $\delta=\left(  \log n\right)  ^{-1}$
as
\begin{equation}
\mathcal{D}=\{d\in\mathbb{R}_{+}^{n}:\left\Vert d\right\Vert ^{2}=1\text{ and
}\sup_{j}d_{j}^{2}\leq\delta/n\rho\}. \label{shrinking_coef}%
\end{equation}
For any test, we are interested in the worst case type II error under the
constraint $f\in\Sigma(\beta,M)\cap B_{\rho}$. A monotonicity argument shows
that for every $\psi_{d}$, this is attained when $\left\Vert f\right\Vert
^{2}$ is minimal, i.e. at $\left\Vert f\right\Vert ^{2}=\rho$. It follows that
for quadratic tests $\psi_{d}$, we may replace the restriction $f\in B_{\rho}$
by $f\in B_{\rho}^{\prime}$ where
\[
B_{\rho}^{\prime}=\{f\in\mathbb{R}^{n}:\rho\leq\left\Vert f\right\Vert
^{2}\leq2\rho\}.
\]
For $f\in\mathbb{R}^{n}$ we set $f^{2}:=\left(  f_{j}^{2}\right)  _{j=1}^{n}$.
For $d\in\mathcal{D}$ and $g\in\mathbb{R}_{+}^{n}$ define the functional
\[
L(d,g)=\frac{n}{\sqrt{2}}\left\langle d,g\right\rangle .
\]

\begin{lemma}
\label{lem-quad-tests-original}(a) Under $H_{0}$, we have $T\rightsquigarrow
N(0,1)$ uniformly over $d\in\mathcal{D}.$\newline(b) The statistic $T$ given
by (\ref{Tn-tilde-def}) fulfills
\[
T-L(d,f^{2})\rightsquigarrow N(0,1)
\]
uniformly over $d\in\mathcal{D}$ and $f\in B_{\rho}^{\prime}.$\newline(c)
Suppose $f$ is random such that $f_{j}\sim N\left(  0,\sigma_{j}^{2}\right)  $
for a certain $\sigma\in\mathbb{R}^{n}$. Then the statistic $T$ given by
(\ref{Tn-tilde-def}) fulfills%
\[
T-L(d,\sigma^{2})\rightsquigarrow N(0,1)
\]
uniformly over $d\in\mathcal{D}$ and $\sigma\in B_{\rho}^{\prime}$.
\end{lemma}

Denote the expectation under the model of (c) by $E_{\sigma}^{\ast}$. The
lemma implies that for uniformly over $d\in\mathcal{D}$ and $f\in\left\{
0\right\}  \cup\left(  \Sigma(\beta,M)\cap B_{\rho}^{\prime}\right)  $
\begin{align}
E_{f}(1-\psi_{d})  &  =\Phi(z_{\alpha}-L(d,f^{2}%
))+o(1)\label{in-terms-of-Phi-1}\\
&  =E_{f}^{\ast}(1-\psi_{d})+o(1). \label{in-terms-of-Phi-2}%
\end{align}
\newline In particular, all quadratic tests $\psi_{d}$ with $d\in\mathcal{D}$
are aymptotic $\alpha$-tests under $H_{0}:f=0$. To characterize the worst case
error under the alternative $H_{a}:f\in\Sigma(\beta,M)\cap B_{\rho}$, we use
(\ref{in-terms-of-Phi-1}) and the strict monotonicity of $\Phi$ and look for a
saddlepoint of the functional $L(d,f^{2})$.

\begin{lemma}
\label{lem-saddle}For $n$ large enough, there exists a saddlepoint $d_{0}%
\in\mathcal{D},f_{0}\in\Sigma(\beta,M)\cap B_{\rho}^{\prime}$ of the
functional $L(d,f^{2})$ such that
\[
L(d,f_{0}^{2})\leq L(d_{0},f_{0}^{2})\leq L(d_{0},f^{2})\text{ }%
\]
for all $d\in\mathcal{D}$ and all $f\in\Sigma(\beta,M)\cap B_{\rho}^{\prime}$.
\end{lemma}

The normal distribution on the signal $f$ postulated in (c) will be
interpreted as a prior distribution. The next result shows that the Bayesian
tests in this context are quadratic tests $\psi_{d}$, and in particular, if
the $\sigma^{2}$ is taken at the saddlepoint ($\sigma_{0}^{2}=f_{0}^{2}$) then
$d\in\mathcal{D}$, i.e. it fulfills the infinitesimality condition $d_{j}%
^{2}\leq\delta/n\rho$.

\begin{lemma}
\label{lem-Bayes-test}(a) For any $\sigma^{2}\in\mathbb{R}_{+}^{n}$, the
Neyman-Pearson $\alpha$-test for simple hypotheses
\begin{align*}
H_{0}  &  :Y_{j}\sim N(0,n^{-1}),j=1,\ldots,n\;\;\;\text{ vs.}\\
H_{a}^{\ast}  &  :Y_{j}\sim N(0,\sigma_{j}^{2}+n^{-1}),j=1,\ldots,n
\end{align*}
\newline is equivalent to a quadratic test of form $\psi_{d}=\mathbf{1}%
\left\{  T>t\right\}  $ where $T=\sum_{j=1}^{n}d_{j}Y_{j}^{2}$, $d\in
\mathbb{R}_{+}^{n}$, $\left\Vert d\right\Vert =1$. \newline(b) If $\sigma
^{2}=f_{0}^{2}$ then the pertaining $d$ is in $\mathcal{D}$ for $n$ large
enough, and $t\rightarrow z_{\alpha}$.
\end{lemma}

Part (b) implies that
\begin{equation}
\inf_{\phi:E_{0}\phi\leq\alpha}E_{f_{0}}^{\ast}(1-\phi)=\inf_{d\in\mathcal{D}%
}E_{f_{0}}^{\ast}(1-\psi_{d})+o(1). \label{bayes-tests-are-linear}%
\end{equation}
We are now ready to present the essence of the argument underlying the result
of Ermakov \cite{Ermak-90}. Recall that $\pi_{n}(\alpha,\rho,\beta,M)$ denotes
the minimax type II error over all $\alpha$-tests. Denote the value of
$L(d,f^{2})$ at the saddlepoint
\begin{equation}
L_{0}:=L(d_{0},f_{0}^{2})=\sup_{d\in\mathcal{D}}\inf_{f\in\Sigma(\beta,M)\cap
B_{\rho}^{\prime}}L_{n}(d,f^{2})=\inf_{f\in\Sigma(\beta,M)\cap B_{\rho
}^{\prime}}\sup_{d\in\mathcal{D}}L_{n}(d,f^{2}). \label{value-game}%
\end{equation}
We begin with an $\alpha^{\prime}>\alpha$ such that asymptotic $\alpha$-tests
are $\alpha^{\prime}$-tests for $n$ large enough. Then
\begin{align}
\pi_{n}(\alpha^{\prime},\rho,\beta,M)  &  =\inf_{\phi:E_{0}\phi\leq
\alpha^{\prime}}\sup_{f\in\Sigma(\beta,M)\cap B_{\rho}}E_{f}\left(
1-\phi\right) \label{begin-reasoning}\\
&  \leq\inf_{d\in\mathcal{D}}\sup_{f\in\Sigma(\beta,M)\cap B_{\rho}}%
E_{f}(1-\psi_{d})\nonumber\\
&  =\inf_{d\in\mathcal{D}}\sup_{f\in\Sigma(\beta,M)\cap B_{\rho}^{\prime}%
}E_{f}(1-\psi_{d})\nonumber\\
&  =\inf_{d\in\mathcal{D}}\sup_{f\in\Sigma(\beta,M)\cap B_{\rho}^{\prime}}%
\Phi(z_{\alpha}-L_{n}(d,f^{2}))+o(1)\text{ [relation (\ref{in-terms-of-Phi-1}%
)]}\nonumber\\
&  =\Phi(z_{\alpha}-L_{n}(d_{0},f_{0}^{2}))+o(1)\text{ [monotonicity of }%
\Phi\text{ and (\ref{value-game})]}\nonumber\\
&  =\inf_{d\in\mathcal{D}}E_{f_{0}}^{\ast}(1-\psi_{d})+o(1)\text{ [relation
(\ref{in-terms-of-Phi-2})]}\nonumber\\
&  =\inf_{\phi:E_{0}\phi\leq\alpha}E_{f_{0}}^{\ast}(1-\phi)+o(1)\text{
[relation (\ref{bayes-tests-are-linear})]. }\nonumber
\end{align}
The main term of the last expression is the Bayes risk for a prior
distribution $f_{j}\sim N(0,f_{0j}^{2})$ in the original model $Y_{j}\sim
N\left(  f_{j},n^{-1}\right)  $. Since $f_{0}\in\Sigma(\beta,M)\cap B_{\rho
}^{\prime}$ and is extremal there, it fulfills
\[
\sum_{j=1}^{n}f_{0j}^{2}j^{2\beta}=M\text{, }\sum_{j=1}^{n}f_{0j}^{2}=\rho
\]
(see the precise description of the saddlepoint $\left(  d_{0},f_{0}\right)  $
in Lemma \ref{lem-descrip-saddlepoint} below). It can therefore be shown that
(as in the original Pinsker \cite{pinsk-80} result) that this prior
distribution asymptotically concentrates on every set of the form
$\Sigma(\beta,M(1+\varepsilon))\cap B_{\rho(1-\varepsilon)}^{\prime}$ for
$\varepsilon>0$. A standard reasoning by truncation shows that in this case,
for a certain probability measure $G$ strictly concentrated on $\Sigma
(\beta,M(1+\varepsilon))\cap B_{\rho(1-\varepsilon)}^{\prime}$
\[
\inf_{\phi:E_{0}\phi\leq\alpha}E_{f_{0}}^{\ast}(1-\phi)\leq\inf_{\phi
:E_{0}\phi\leq\alpha}\int E_{f}(1-\phi)dG(f)+o(1).
\]
However, by the relation between Bayes and minimax risk
\begin{equation}
\inf_{\phi:E_{0}\phi\leq\alpha}\int E_{f}(1-\phi)dG(f)\leq\pi_{n}(\alpha
,\rho(1-\varepsilon),\beta,M(1+\varepsilon)). \label{from-bayes-to-minimax}%
\end{equation}
Summarizing (\ref{begin-reasoning})-(\ref{from-bayes-to-minimax}) we have
obtained for every $\varepsilon>0$%
\[
\pi_{n}(\alpha(1+\varepsilon),\rho,\beta,M)\leq\Phi(z_{\alpha}-L_{n}%
(d_{0},f_{0}^{2}))+o(1)\leq\pi_{n}(\alpha,\rho(1-\varepsilon),\beta
,M(1+\varepsilon))+o(1)
\]
Below in Lemma \ref{lem-L0-value} is it shown that if $\rho=c\cdot
n^{-4\beta/(4\beta+1)}$, $c$ constant then
\[
L(d_{0},f_{0}^{2})\sim\sqrt{A_{0}M^{-1/(2\beta)}c^{2+1/(2\beta)}/2}.
\]
Since the right side is continuous in $M$ and $c$ , the result of Proposition
\ref{propos-ermak-90} follows.

\section{Proof of Theorem \ref{adaptationimpossible}%
\label{sec:adaptation-impossible}}

For brevity we write $A_{i}=A(c,\beta,M_{i}),i=1,2$ in this section. Assume
there exists a test $\phi_{n}$ not depending on on $M$ such that
\begin{align}
&  E_{0,n}\phi_{n}\leq\alpha+o(1),\label{upperbound}\\
\sup_{f\in\Sigma(\beta,M_{i})\cap B_{\rho}}  &  E_{f,n}(1-\phi_{n})\leq
\Phi(z_{\alpha}-\sqrt{A_{i}/2})+o(1), \label{power}%
\end{align}
for $i=1$ or $2$. Let $G_{n,M_{i}}$ be the Gaussian prior for $f$ with
$f_{j}\sim N(0,\sigma_{j}^{\ast2})$ independently, where
\[
\sigma_{j}^{\ast2}(M_{i})=(\lambda-\mu j^{2\beta})_{+},\ \ j=1,2,\ldots
\]
and where $\lambda$ and $\mu$ are determined by
\[
\sum j^{2\beta}\sigma_{j}^{\ast2}=M_{i}\quad\text{ and }\quad\sum\sigma
_{j}^{\ast2}=\rho.
\]
It can be shown that $G_{n,M_{i}}$ asymptotically concentrates on
$\Sigma(\beta,M_{i}\left(  1+\varepsilon\right)  )\cap B_{\rho\left(
1-\varepsilon\right)  }^{\prime}$ for any small $\varepsilon>0$. Then
\[
\sup_{\Sigma(\beta,M_{i}\left(  1+\varepsilon\right)  )\cap B_{\rho\left(
1-\varepsilon\right)  }^{\prime}}E_{f,n}(1-\phi_{n})\geq(1+o(1))\cdot\int
E_{f,n}(1-\phi_{n})\,G_{n,M_{i}}(df).
\]
Recall $Y_{j}=f_{j}+n^{-1/2}\xi_{j}$. Let the joint distributions of
$(Y_{j})_{0}^{\infty}$ under the priors $G_{n,0},G_{n,M_{1}}$ and $G_{n,M_{2}%
}$ be $Q_{0,n},Q_{1,n}$ and $Q_{2,n}$, respectively, i.e.,
\begin{align*}
&  Q_{0,n}:Y_{j}\sim N(0,n^{-1}),\ j=1,2,\ldots\\
&  Q_{1,n}:Y_{j}\sim N(0,n^{-1}+\sigma_{j}^{\ast2}(M_{1})),\ j=1,2,\ldots\\
&  Q_{2,n}:Y_{j}\sim N(0,n^{-1}+\sigma_{j}^{\ast2}(M_{2})),\ j=1,2,\ldots
\end{align*}
Therefore,
\begin{align*}
&  E_{Q_{0,n}}\phi_{n}=E_{0,n}\phi_{n},\\
&  E_{Q_{i,n}}(1-\phi_{n})=\int E_{f,n}(1-\phi_{n})\,G_{n,M_{i}}%
(df),\ \ i=1,2.
\end{align*}
Combining these with (\ref{power}) and (\ref{upperbound}) gives%
\[
E_{Q_{0,n}}\phi_{n}\leq\alpha+o(1),
\]%
\begin{align*}
E_{Q_{i,n}}(1-\phi_{n})  &  \leq\Phi\left(  z_{\alpha}-\sqrt{A_{i}/2}\right)
\\
&  +\left\vert \sup_{f\in\Sigma(\beta,M_{i}\left(  1+\varepsilon\right)  )\cap
B_{\rho\left(  1-\varepsilon\right)  }^{\prime}}E_{f,n}(1-\phi_{n})-\sup
_{f\in\Sigma(\beta,M_{i})\cap B_{\rho}}E_{f,n}(1-\phi_{n})\right\vert +o(1).
\end{align*}
Note that $E_{f,n}(1-\phi_{n})$ is continuous in $f$. Since $\varepsilon$ can
be arbitrarily small, we have%
\[
E_{Q_{i,n}}(1-\phi_{n})\leq\Phi\left(  z_{\alpha}-\sqrt{A_{i}/2}\right)
+o(1),\ \ i=1,2.
\]
The likelihood ratio of $Q_{i,n}$ against $Q_{0,n}$ is
\begin{align*}
\frac{dQ_{i,n}}{dQ_{0,n}}  &  =\exp\left(  -\frac{1}{2}\sum_{j}\left(
\frac{Y_{j}^{2}}{n^{-1}+\sigma_{j}^{\ast2}(M_{i})}-\frac{Y_{j}^{2}}{n^{-1}%
}\right)  \right)  \cdot\prod_{j}\left(  \frac{n^{-1}}{n^{-1}+\sigma_{j}%
^{\ast2}(M_{i})}\right)  ^{1/2}\\
&  =\exp\left(  \frac{1}{2}\sum_{j}\frac{n^{2}\sigma_{j}^{\ast2}(M_{i}%
)}{1+n\sigma_{j}^{\ast2}(M_{i})}Y_{j}^{2}\right)  \cdot\prod_{j}\left(
\frac{n^{-1}}{n^{-1}+\sigma_{j}^{\ast2}(M_{i})}\right)  ^{1/2}.
\end{align*}
Therefore, by the factorization theorem, it is seen that the bivariate vector
\[
T_{n}=\left(  \sum_{j}\frac{n^{2}\sigma_{j}^{\ast2}(M_{1})(Y_{j}^{2}-n^{-1}%
)}{(1+n\sigma_{j}^{\ast2}(M_{1}))\sqrt{2n^{2}\sum_{k}\sigma_{k}^{\ast4}%
(M_{1})}},\sum\frac{n^{2}\sigma_{j}^{\ast2}(M_{2})(Y_{j}^{2}-n^{-1}%
)}{(1+n\sigma_{j}^{\ast2}(M_{2}))\sqrt{2n^{2}\sum_{k}\sigma_{k}^{\ast4}%
(M_{2})}}\right)
\]
is a sufficient statistic for the family of distributions $\{Q_{0,n}%
,Q_{1,n},Q_{2,n}\}$. Write the induced family for $T_{n}$ as $\{Q_{0,n}%
^{T},Q_{1,n}^{T},Q_{2,n}^{T}\}$ and take the conditional expectation $\phi
_{n}^{\ast}(T_{n})=E_{Q_{i,n}}(\phi_{n}|T_{n})$. By sufficiency the (possibly
randomized) test $\phi_{n}^{\ast}(T_{n})$ for $\{Q_{0,n}^{T},Q_{1,n}%
^{T},Q_{2,n}^{T}\}$ is as good as $\phi_{n}$ (cf. for instance Theorem 4.66 in
\cite{Liese-book}), that is
\begin{align}
&  E_{Q_{0,n}^{T}}\phi_{n}^{\ast}=E_{0,n}\phi_{n}\leq\alpha+o(1),\\
&  E_{Q_{i,n}^{T}}(1-\phi_{n}^{\ast})=E_{Q_{1,n}}\phi_{n}\leq\Phi(z_{\alpha
}-\sqrt{A_{i}/2})+o(1),\ \ i=1,2. \label{UMPat}%
\end{align}
Then we have the following lemma, which is proved later.

\begin{lemma}
\label{lem:LimitExp}Under $\{Q_{0,n},Q_{1,n},Q_{2,n}\}$, the law of the
statistic $T_{n}$ converges in total variation to $N(0,\Sigma)$, $N(\mu
_{1},\Sigma)$ and $N(\mu_{2},\Sigma)$ respectively, where
\begin{align}
&  \mu_{1}=(\sqrt{A_{1}/2},r\sqrt{A_{1}/2})^{\prime},\nonumber\\
&  \mu_{2}=(r\sqrt{A_{2}/2},\sqrt{A_{2}/2})^{\prime},\nonumber\\
&  \Sigma=\left(
\begin{array}
[c]{cc}%
1 & r\\
r & 1
\end{array}
\right)  ,\nonumber\\
&  r=\left(  \frac{M_{1}}{M_{2}}\right)  ^{1/(4\beta)}\cdot\frac
{4\beta+1-M_{1}/M_{2}}{4\beta}.
\end{align}

\end{lemma}

Then by the weak compactness theorem (c.f. \cite{Lehm-roman-book}, A.5.1 ),
there exists a test $\phi^{\ast}$ and a subsequence $\phi_{n_{k}}^{\ast}$ such
that $\phi_{n_{k}}^{\ast}$ converges weakly to $\phi^{\ast}$. Thus
\begin{align*}
&  E_{Q_{0,n}^{T}}\phi^{\ast}\leq\alpha,\\
&  E_{Q_{i,n}^{T}}(1-\phi^{\ast})\leq\Phi(z_{\alpha}-\sqrt{A_{i}%
/2}),\ \ i=1,2.
\end{align*}
For $i=1,2$ respectively, by the Neyman-Pearson lemma and some direct
calculations, the right hand side of the previous inequality is the type II
error of the uniformly most powerful test for $N(0,\Sigma)$ against $N(\mu
_{i},\Sigma)$. Therefore, $\phi^{\ast}$ is a uniformly most powerful test for
$N(0,\Sigma)$ against $\{N(\mu_{1},\Sigma),N(\mu_{2},\Sigma)\}$.

Note that $r$ in Lemma \ref{lem:LimitExp} is monotone increasing with respect
to $M_{1}/M_{2}$, and then $0<r<1$ for $M_{2}>M_{1}>0$. Thus, $\mu_{1}$,
$\mu_{2}$ and the origin are not on the same line. For $i=1,2$ respectively,
the log-likelihood ratio for $N(\mu_{i},\Sigma)$ against $N(0,\Sigma)$ is
$T^{\prime-1}\mu_{i}=T_{i}\cdot A_{i}$. Then by the necessity part of the
Neyman-Pearson lemma (\cite{Lehm-roman-book}, Theorem 3.2.1), the uniformly
most powerful test for $N(0,\Sigma)$ against $N(\mu_{i},\Sigma)$ has the form
of $\mathbf{1}\{T_{i}>k_{i}\}$. But since these two types of tests can never
coincide, there is no uniformly most powerful test for $N(0,\Sigma)$ against
$\{N(\mu_{1},\Sigma),N(\mu_{2},\Sigma)\}$. By this contradiction, Theorem
\ref{adaptationimpossible} is proved.

\begin{proof}
[Proof of Lemma \textbf{\ref{lem:LimitExp}}]For simplicity, we only show the
result for the first coordinate of $T_{n}$. The proof can be extended to
$T_{n}$ naturally. Under $Q_{0,n}$, the characteristic function of
$\frac{n(Y_{j}^{2}-1/n)}{\sqrt{2}}\sim N(0,1)$ is $g(t)=\exp(-t^{2}/2)$. Note
$g(t)=1-\frac{1}{2}t^{2}+o(t^{2})$, as $t\rightarrow0$ and $\int|g(t)|<\infty
$. The density of $T_{n,1}$ can be written as
\[
p_{n}(x)=\frac{1}{2\pi}\int e^{-itx}\prod g\left(  \frac{\sigma_{j}^{\ast
2}(M_{1})\cdot t}{(1+n\sigma_{j}^{\ast2}(M_{1}))\sqrt{\sum_{k}\sigma_{k}%
^{\ast4}(M_{1})}}\right)  ,
\]
where, by Levy's continuity theorem, the integrand converges to $e^{-itx}%
\exp\{-t^{2}/2\}$. By splitting the integral into two parts and using
dominated convergence, it can be shown that the integral converges to
\[
\frac{1}{2\pi}\int e^{-itx}e^{-t^{2}/2}\,dt=\frac{e^{-x^{2}/2}}{\sqrt{2\pi}}.
\]
Then an application of Scheff\'{e}'s theorem (cf. \cite{VdVaart-book}, 2.30)
establishes convergence in total variation. The correlation $r$ can be
calculated directly.
\end{proof}

\begin{privatenotes}
\begin{align*}
r&  =\frac{\sum\sigma_{j}^{*2}(M_{1}) \sigma_{j}^{*2}(M_{2})}{\sqrt
{\sum\sigma_{j}^{*4}(M_{1})\cdot
\sum\sigma_{j}^{*4}(M_{2})}}\nonumber\\
&  \sim\frac{\left(  \frac{N(M_{1})}{N(M_{2})}\right)  ^{1/2}\frac{1}%
{N(M_{1})}\sum\left(  1-(j/N(M_{1}))^{2\beta}\right)  _{+}\left(
1-(j/N(M_{2}))^{2\beta}\right)  }{\left(  \frac{1}{N(M_{1})}\sum
(1-(j/N(M_{1}))^{2\beta})_{+}\cdot\frac{1}{N(M_{2})}\sum(1-(j/N(M_{1}%
))^{2\beta})_{+}\right)  ^{1/2}}\nonumber\\
&  \sim\frac{(M_{1}/M_{2})^{1/(4\beta)}\int_{0}^{\infty}(1-t^{2\beta}%
)_{+}(1-M_{1}t^{2\beta}/M_{2})_{+}dt}{\int_{0}^{\infty}(1-t^{2\beta})_{+}%
\,dt}\nonumber\\
\end{align*}
\end{privatenotes}


\section{Proof of Theorem \ref{sharpadaptation}\label{sec:sharp-adaptation}}

Choose $\tilde{N}$ and $\gamma_{n}=o(1)$ such that
\begin{equation}
\gamma_{n}^{1/2\beta}\cdot n^{2/(4\beta+1)}\gg\tilde{N}\gg c_{n}^{-1/2\beta
}\cdot n^{2/(4\beta+1)}, \label{tildeN}%
\end{equation}
e.g. $\gamma_{n}=c_{n}^{-1/2}$, $\tilde{N}=c_{n}^{-1/3\beta}\cdot
n^{2/(4\beta+1)}$. Define
\begin{align*}
M_{0}  &  =M_{0}(f)=\sum_{j=1}^{\tilde{N}}j^{2\beta}f_{j}^{2}+\gamma_{n},\\
N  &  =N(M_{0})=\left(  \frac{(4\beta+1)M_{0}}{\rho}\right)  ^{1/2\beta},\\
\tilde{\lambda}  &  =\tilde{\lambda}(M_{0})=\frac{2\beta+1}{2\beta}\left(
\frac{1}{M_{0}(4\beta+1)}\right)  ^{1/(2\beta)}\rho^{(2\beta+1)/2\beta},\\
\tilde{d}_{j}^{{}}  &  =\tilde{d}_{j}(M_{0})=\tilde{\lambda}[1-(j/N)^{2\beta
}]_{+},
\end{align*}
which all depend on the unknown $f$. Define the oracle statistic
\[
T_{n}^{\ast}=\frac{n^{2}\sum_{j}\tilde{d}_{j}(M_{0})Y_{j}^{2}-n\sum_{j}%
\tilde{d}_{j}(M_{0})}{\sqrt{2n^{2}\sum_{j}\tilde{d}_{j}^{2}(M_{0})}},
\]
and the oracle test $\phi_{n}^{\ast}=\mathbf{1}\{T_{n}^{\ast}>z_{\alpha}\}$.
The following lemma holds; it is proved later.

\begin{lemma}
Under the assumptions of Theorem \ref{sharpadaptation}, the oracle test
$\phi_{n}^{\ast}$ is an asymptotic $\alpha$-test and
\[
\limsup_{n}\frac{1}{c_{n}^{2+1/(2\beta)}}\log\Psi(\phi_{n}^{\ast},\rho
_{n},\beta,M)\leq-\frac{A_{0}(\beta)M^{-1/2\beta}}{4}%
\]
\label{lem:oracle}
\end{lemma}

Define
\[
\hat{M}=\sum_{j=1}^{\tilde{N}}(Y_{j}^{2}-1/n)j^{2\beta}+\gamma_{n}%
\]
and introduce the statistic
\[
T_{n}=\frac{n^{2}\sum\tilde{d}_{j}(\hat{M})Y_{j}^{2}-n\sum\tilde{d}_{j}%
(\hat{M})}{\sqrt{2n^{2}\sum\tilde{d}_{j}^{2}(\hat{M})}}%
\]
and also the test
\[
\phi_{n}=\mathbf{1}\{T_{n}>z_{\alpha}\}.
\]
For $\hat{M}$, we have the following lemma, which is proved later.

\begin{lemma}
Under the assumptions of Theorem \ref{sharpadaptation}, we have
\[
\frac{\hat{M}}{M_{0}(f)}-1=o_{p}(1),
\]
uniformly for $f\in\Sigma(\beta,M)\cap B_{\rho}$. \label{lem:Mconsist}
\end{lemma}

Now rewrite
\[
T_{n}=\sum_{j}\frac{\tilde{d}_{j}(\hat{M})}{\sqrt{\sum\tilde{d}_{j}^{2}%
(\hat{M})}}\cdot\frac{Y_{j}^{2}-1/n}{\sqrt{2n^{-2}}},
\]
where $\tilde{d}_{j}(\hat{M})=\tilde{\lambda}(1-(j/N(\hat{M}))^{2\beta})_{+}$.
Since $\tilde{\lambda}$ in the last display can be canceled, for simplicity we
write $\tilde{d}_{j}(\hat{M})=(1-(j/N(\hat{M}))^{2\beta})_{+}$ from now on in
this section. First, since $N(\hat{M})\geq N(\gamma_{n})$, we have
\begin{align*}
\sum\tilde{d}_{j}^{2}(\hat{M})  &  =\sum\left(  1-\left(  \frac{j}{N(\hat{M}%
)}\right)  ^{2\beta}\right)  _{+}^{2}\\
&  \sim N(\hat{M})\int_{0}^{1}(1-t^{2\beta})_{+}^{2}\,dt\\
&  =N(\hat{M})K(\beta).
\end{align*}
\begin{privatenotes}
{\bf Private Notes. }
Need $N(\hat{M})\rightarrow\infty$.
\end{privatenotes}Therefore,
\[
T_{n}=(1+o(1))\sum\frac{\tilde{d}_{j}(\hat{M})}{\sqrt{N(\hat{M})K(\beta)}%
}\cdot\frac{Y_{j}^{2}-1/n}{\sqrt{2n^{-2}}}.
\]
By Lemma \ref{lem:Mconsist},
\[
T_{n}=(1+o(1))\sum_{j}\frac{\tilde{d}_{j}(\hat{M})}{\sqrt{N(M_{0}(f))K(\beta
)}}\cdot\frac{Y_{j}^{2}-1/n}{\sqrt{2n^{-2}}}.
\]
At this point, make $\hat{M}$ independent of $Y_{j}^{2}$ by sample splitting.
Set $n=\tau n+(1-\tau)n$, where $\tau$ is close to $1$ but fixed, and
$n_{1}=\tau n,n_{2}=(1-\tau)n$. Assume two sets of observations
\begin{align}
&  Y_{1j}=f_{j}+n_{1}^{-1/2}\xi_{1j},j=1,2,\ldots\\
&  Y_{2j}=f_{j}+n_{2}^{-1/2}\xi_{2j},j=1,2,\ldots
\end{align}
Use $\{Y_{2j}\}$ to obtain $\hat{M}$, and now replace $T_{n}$ by
\[
T_{n}^{s}=(1+o(1))\sum_{j}\frac{\tilde{d}(\hat{M})}{\sqrt{N(M_{0}(f))K(\beta
)}}\cdot\frac{Y_{1j}^{2}-n^{-1}}{\sqrt{2}n^{-1}}.
\]
Denote the difference of coefficients by $\Delta_{j}=\tilde{d}_{j}(\hat
{M})-\tilde{d}_{j}(M_{0}(f))$. Note the largest difference is obtained at
$j\approx\min\{N(\hat{M}),N(M_{0}(f))\}$. Then
\[
|\Delta_{j}|\leq\frac{|\hat{M}-M_{0}(f)|}{\gamma_{n}}%
\]
uniformly for all $j$. Note in $T_{1}$ there are at most $C_{2}c_{n}%
^{-1/(2\beta)}n^{2/(4\beta+1)}$ nonzero coefficients. Then
\[
T_{n}^{s}=(1+o(1))\sum_{j=1}^{C_{2}c_{n}^{-1/(2\beta)}n^{2/(4\beta+1)}}%
\frac{\tilde{d}_{j}(M_{0}(f))}{\sqrt{N(M_{0}(f))K(\beta)}}\eta_{j}+r_{n}%
\]
where $\eta_{j}=\frac{Y_{1j}^{2}-n_{1}^{-1}}{\sqrt{2}n_{1}^{-1}}$, and
\[
r_{n}=\sum_{j=1}^{C_{2}c_{n}^{-1/(2\beta)}n^{2/(4\beta+1)}}\frac{\Delta
_{j}\eta_{j}}{\sqrt{N(M_{0}(f))K(\beta)}}.
\]
Under $H_{0}$, the r.v.%
\'{}%
s $\eta_{j}$ are independent of $\hat{M}$ and $E\eta_{j}=0$, $\text{Var}%
(\eta_{j})=1$. Thus $\mathrm{Var}(r_{n})=Er_{n}^{2}=EE(r_{n}^{2}|\{Y_{2j}\})$
and
\[
E(r_{n}^{2}|\{Y_{2j}\})=E\sum_{j=1}^{C_{2}c_{n}^{-1/(2\beta)}n^{2/(4\beta+1)}%
}\frac{\Delta_{j}^{2}}{N(M_{0}(f))K(\beta)}\leq\frac{|\hat{M}-M_{0}(f)|^{2}%
}{\gamma_{n}^{2+1/(2\beta)}}.
\]
Therefore, by the result for $\mathrm{Var}(\hat{M})$ in the proof of Lemma
\ref{lem:Mconsist},
\[
\mathrm{Var}(r_{n})\leq\frac{E|\hat{M}-M_{0}(f)|^{2}}{\gamma_{n}%
^{2+1/(2\beta)}}=\frac{\mathrm{Var}(\hat{M})}{\gamma_{n}^{2+1/(2\beta)}}%
\leq\frac{2K(\beta)\tilde{N}^{4\beta+1}}{n^{2}\gamma_{n}^{2+1/(2\beta)}}%
+\frac{4\tilde{N}^{2\beta}M}{n\gamma_{n}^{2+1/(2\beta)}},
\]
where the last two terms converge to $0$ by the first inequality in
(\ref{tildeN}). Hence, under $H_{0}$, the r.v.%
\'{}%
s $T_{n}$ and $T_{n}^{s}$ converge to $N(0,1)$ in law.

Next, we consider $T_{n}$ or $T_{n}^{s}$ under the alternative. The worst case
type II error is determined by the following quantity
\[
L_{n}=\frac{n}{\sqrt{2}}\inf_{f\in\Sigma(\beta,M)\cap B_{\rho}}\frac
{\sum_{j=1}^{\tilde{N}}f_{j}^{2}\tilde{d}_{j}(\hat{M})}{\left(  \sum\tilde
{d}_{j}(\hat{M})\right)  ^{1/2}}.
\]
First, since $N(\hat{M})\geq\left(  \frac{\gamma_{n}}{c_{n}}\right)
^{1/(2\beta)}\cdot n^{2/(4\beta+1)}\rightarrow\infty$,
\begin{align}
\tilde{d}_{j}^{2}  &  =\sum_{j=1}^{\tilde{N}}\left(  1-(j/N)^{2\beta}\right)
_{+}^{2}\nonumber\\
&  =(1+o(1))N\int_{0}^{1}(1-t^{2\beta})^{2}dt\nonumber\\
&  =(1+o(1))N\cdot\frac{8\beta^{2}}{(2\beta+1)(4\beta+1)}. \label{dj2}%
\end{align}
Second, consider
\[
\sum_{j=1}^{\tilde{N}}f_{j}^{2}\tilde{d}_{j}(\hat{M})=\sum_{j=1}^{\tilde{N}%
}f_{j}^{2}(1-(j/N)^{2\beta})_{+}.
\]
Note
\begin{align}
\sum_{j=1}^{\tilde{N}}f_{j}^{2}  &  =\sum_{j=1}^{\infty}f_{j}^{2}%
-\sum_{j=\tilde{N}+1}^{\infty}f_{j}^{2}\nonumber\\
&  \geq\rho-\tilde{N}^{-2\beta}M\nonumber\\
&  =\rho\left(  1-\frac{M}{\rho\tilde{N}^{2\beta}}\right) \nonumber\\
&  =\rho(1+o(1)), \label{fj2}%
\end{align}
where the last step is refers to the second inequality of (\ref{tildeN}). On
the other hand, since $\tilde{N}\gg N$ and $N(\hat{M})=[(4\beta+1)\hat{M}%
\rho^{-1}]^{1/(2\beta)}$,
\begin{align}
\sum_{j=1}^{N}f_{j}^{2}(j/N)^{2\beta}+\sum_{j=N+1}^{\tilde{N}}f_{j}^{2}  &
\leq\sum_{j=1}^{\tilde{N}}f_{j}^{2}(j/N)^{2\beta}\nonumber\\
&  \leq N^{-2\beta}M_{0}(f)\nonumber\\
&  =\rho(1+4\beta)^{-1}. \label{fjj2}%
\end{align}
Combining (\ref{fjdj})-(\ref{fjj}) gives
\[
\sum_{j=1}^{\tilde{N}}f_{j}^{2}\tilde{d}_{j}\geq(1+o(1))\tilde{\lambda}%
\rho\cdot\frac{4\beta}{4\beta+1}.
\]
Combining this with (\ref{dj}) gives
\begin{align*}
\frac{n\sum_{j=1}^{\tilde{N}}f_{j}^{2}\tilde{d}_{j}}{(2\sum\tilde{d}_{j}%
^{2})^{1/2}}  &  \geq(1+o(1))\frac{n}{\sqrt{2}}\sqrt{\frac{2(2\beta+1)}%
{4\beta+1}\rho^{2}/N}\\
&  \geq(1+o(1))\sqrt{\frac{(2\beta+1)c_{n}^{2+1/(2\beta)}}{(4\beta
+1)^{1+1/(2\beta)}(M+\gamma_{n})^{1/(2\beta)}}}\\
&  \geq(1+o(1))\sqrt{\frac{1}{2}A_{0}(\beta)c_{n}^{2+1/(2\beta)}%
M^{-1/(2\beta)}}%
\end{align*}
Theorem \ref{sharpadaptation} is proved.

\begin{proof}
[Proof of Lemma \ref{lem:oracle}]Rewrite
\[
T_{n}^{\ast}=\sum_{j}\frac{\tilde{d}_{j}(M_{0}(f))}{\sqrt{\sum\tilde{d}%
_{j}^{2}(M_{0}(f))}}\cdot\frac{Y_{j}^{2}-1/n}{\sqrt{2n^{-2}}}.
\]
Under $H_{0}$, we have $f=0$, and $M_{0}(f)=\gamma_{n}$. Since
\[
\sum[1-(j/N)^{2\beta}]_{+}^{2}\sim N\cdot\int_{0}^{1}(1-t^{2\beta}%
)^{2}\,dt=K(\beta)\cdot(\gamma_{n}/c_{n})^{1/2\beta}n^{2/(4\beta+1)},
\]
then
\[
\left\vert \frac{\tilde{d}_{j}(M_{0}(f))}{\sqrt{\sum\tilde{d}_{j}^{2}%
(M_{0}(f))}}\right\vert \leq\frac{1}{\sqrt{K(\beta)\cdot(\gamma_{n}%
/c_{n})^{1/2\beta}n^{2/(4\beta+1)}}}=o(1),
\]
uniformly for all $j$. It can be shown
that $T_{n}^{\ast}$ converges to $N(0,1)$ in law.

By similar arguments, the worst type II error is $(1+o(1))\Phi(z-L_{n})$
where
\[
L_{n}=\inf_{f\in\Sigma(\beta,M)\cap B_{\rho}}\frac{n\sum f_{j}^{2}\tilde
{d}_{j}}{(2\sum\tilde{d}_{j}^{2})^{1/2}}.
\]
Note $\tilde{d}_{j}=\tilde{d}_{j}(M_{0}(f))$ depending on $f$. By the second
inequality of (\ref{tildeN}), we have $\tilde{N}\gg N(M_{0}(f))$ and
$\tilde{d}_{j}=0,$ for $j\geq\tilde{N}$,
\[
L_{n}=\frac{n}{\sqrt{2}}\inf_{f\in\Sigma(M)\cap B_{\rho}}\frac{\sum
_{j=1}^{\tilde{N}}f_{j}^{2}\tilde{d}_{j}}{(\sum\tilde{d}_{j}^{2})^{1/2}}.
\]
First, since $N(M_{0}(f))\geq\left(  \frac{\gamma_{n}}{c_{n}%
}\right)  ^{1/(2\beta)}\cdot n^{2/(4\beta+1)}\rightarrow\infty$ uniformly for
$f\in\Sigma(\beta,M)\cap B_{\rho}$,
\begin{align}
\tilde{d}_{j}^{2}  &  =\tilde{\lambda}^{2}\sum_{j=1}^{\tilde{N}}\left(
1-(j/N)^{2\beta}\right)  _{+}^{2}\nonumber\\
&  =(1+o(1))\tilde{\lambda}^{2}N\int_{0}^{1}(1-t^{2\beta})^{2}dt\nonumber\\
&  =(1+o(1))\tilde{\lambda}^{2}N\cdot\frac{8\beta^{2}}{(2\beta+1)(4\beta+1)},
\label{dj}%
\end{align}
uniformly for $f\in\Sigma(\beta,M)\cap B_{\rho}$. Second, consider
\begin{equation}
\sum_{j=1}^{\tilde{N}}f_{j}^{2}\tilde{d}_{j}=\tilde{\lambda}\sum_{j=1}%
^{\tilde{N}}f_{j}^{2}(1-(j/N)^{2\beta})_{+}=\tilde{\lambda}\left[  \sum
_{j}^{\tilde{N}}f_{j}^{2}-\left(  \sum_{j}^{N}f_{j}^{2}(j/N)^{2\beta}%
+\sum_{j=N+1}^{\tilde{N}}f_{j}^{2}\right)  \right]  . \label{fjdj}%
\end{equation}
Note
\begin{align}
\sum_{j=1}^{\tilde{N}}f_{j}^{2}  &  =\sum_{j=1}^{\infty}f_{j}^{2}%
-\sum_{j=\tilde{N}+1}^{\infty}f_{j}^{2}\nonumber\\
&  \geq\rho-\tilde{N}^{-2\beta}M\nonumber\\
&  =\rho\left(  1-\frac{M}{\rho\tilde{N}^{2\beta}}\right) \nonumber\\
&  =\rho(1+o(1)), \label{fj}%
\end{align}
where the last step is due to the second inequality of (\ref{tildeN}). On the
other hand, since $\tilde{N}\gg N$ and $N=[\rho^{-1}(4\beta+1)M_{0}%
(f)]^{1/(2\beta)}$,
\begin{align}
\sum_{j=1}^{N}f_{j}^{2}(j/N)^{2\beta}+\sum_{j=N+1}^{\tilde{N}}f_{j}^{2}  &
\leq\sum_{j=1}^{\tilde{N}}f_{j}^{2}(j/N)^{2\beta}\nonumber\\
&  \leq N^{-2\beta}M_{0}(f)\nonumber\\
&  =\rho(1+4\beta)^{-1} \label{fjj}%
\end{align}
Combining (\ref{fjdj})-(\ref{fjj}) gives
\[
\sum_{j=1}^{\tilde{N}}f_{j}^{2}\tilde{d}_{j}\geq(1+o(1))\tilde{\lambda}%
\rho\cdot\frac{4\beta}{4\beta+1}%
\]
uniformly for $f\in\Sigma(\beta,M)\cap B_{\rho}$. Combining this with
(\ref{dj}) gives
\begin{align*}
\frac{n\sum_{j=1}^{\tilde{N}}f_{j}^{2}\tilde{d}_{j}}{(2\sum\tilde{d}_{j}%
^{2})^{1/2}}  &  \geq(1+o(1))\frac{n}{\sqrt{2}}\sqrt{\frac{2(2\beta+1)}%
{4\beta+1}\rho^{2}/N}\\
&  \geq(1+o(1))\sqrt{\frac{(2\beta+1)c_{n}^{2+1/(2\beta)}}{(4\beta
+1)^{1+1/(2\beta)}(M+\gamma_{n})^{1/(2\beta)}}}\\
&  \geq(1+o(1))\sqrt{\frac{(2\beta+1)c_{n}^{2+1/(2\beta)}}{(4\beta
+1)^{1+1/(2\beta)}M^{1/(2\beta)}}},
\end{align*}
uniformly for $f\in\Sigma(\beta,M)\cap B_{\rho}$. Therefore,
\[
L_{n}\geq(1+o(1))\sqrt{\frac{1}{2}A_{0}(\beta)c_{n}^{2+1/(2\beta
)}M^{-1/(2\beta)}},
\]
and the result follows. 

\end{proof}

\begin{proof}
[Proof of Lemma \ref{lem:Mconsist}]Since
\begin{align*}
\mathrm{Var}(\hat{M})  &  =\sum_{j=1}^{\tilde{N}}\left(  \frac{2}{n^{2}}%
+\frac{4f_{j}^{2}}{n}\right)  j^{4\beta}\\
&  \leq(1+o(1))\frac{2K(\beta)\tilde{N}^{4\beta+1}}{n^{2}}+\frac{4\tilde
{N}^{2\beta}M}{n},
\end{align*}
by the first inequality of (\ref{tildeN}),
\[
\frac{\mathrm{Var}(\hat{M})}{\gamma_{n}^{2}}=o(1)
\]
uniformly for $f\in\Sigma\cap V_{\rho}$. Combining with $E\hat{M}=M_{0}(f)$
and using Chebyshev's inequality give
\[
\frac{\left\vert \hat{M}-M_{0}(f)\right\vert }{\gamma_{n}}=o_{p}(1),
\]
and then
\[
\left\vert \frac{\hat{M}}{M_{0}(f)}-1\right\vert \leq\frac{\left\vert \hat
{M}-M_{0}(f)\right\vert }{\gamma_{n}}=o_{p}(1),
\]
uniformly for $f\in\Sigma\cap V_{\rho}$.
\end{proof}


\section{Appendix\label{sec: Appendix}}

\subsection{Adaptive minimax estimation with known $\mathbf{\beta}%
$\label{sec:adaptive-estim}}

For the convenience of the reader, we sketch the modified plug-in method of
Golubev \cite{Golub-87} allowing to attain the Pinsker bound for known
smoothness $\beta$ and unknown bound $M$, in the framework of Sobolev
ellipsoids. For more comprehensive results, allowing also for unknown $\beta$,
cf. \cite{gol-90-quasi}, \cite{Tsyb-book}. Consider the estimation problem for
$f=(f_{j})_{j=1}^{\infty}$, with squared $l_{2}$-loss, in the Gaussian
sequence model
\[
Y_{j}=f_{j}+n^{-1/2}\xi_{j}%
\]
with $f\in\Sigma\left(  \beta,M\right)  $. With known $\beta$ and unknown $M$,
the aim is to find an estimator which is asymptotically minimax in the sense
of Pinsker \cite{pinsk-80}. For known $M$, the optimal filter coefficients are
$(1-\mu j^{\beta})_{+}$, where $\mu$ is determined by
\[
\frac{1}{n}\sum j^{\beta}(1-\mu j^{\beta})_{+}=\mu M.
\]
Since
\[
\mu\sim\left(  \frac{\beta\cdot n^{-1}}{M(\beta+1)(2\beta+1)}\right)
^{\beta/(2\beta+1)},
\]
the optimal truncation index (or bandwidth) is of the order $n^{1/(2\beta+1)}%
$.

Choose $n^{1/(2\beta+1/2)}\gg\tilde{N}\gg n^{1/(2\beta+1)}$ and $1\gg
\gamma_{n}\gg\tilde{N}^{2\beta+1/2}/n$, and define
\[
M_{0,f}=\sum_{j=1}^{\tilde{N}}j^{2\beta}f_{j}^{2}+\gamma_{n}.
\]
Define $N=N(M_{0,f})=\alpha\cdot n^{1/(2\beta+1)}M_{0,f}^{1/(2\beta+1)}$,
where $\alpha$ is a constant to be chosen. Define "oracle" filter
coefficients, depending on $f$, as
\[
d_{j}=d(j/N),\text{ where }d(t)=\left(  1-t^{\beta}\right)  _{+}.
\]
Consider the oracle estimator $(d_{j}Y_{j})_{1}^{\infty}$. Its risk is
\begin{align*}
&  \sum(1-d_{j})^{2}f_{j}^{2}+\frac{1}{n}\sum d_{j}^{2}\\
=  &  \sum_{j=1}^{\tilde{N}}(1-d_{j})^{2}f_{j}^{2}+\sum_{j>{\tilde{N}}%
}(1-d_{j})^{2}f_{j}^{2}+\frac{1}{n}\sum d_{j}^{2}\\
:=  &  A_{1}+A_{2}+A_{3}.
\end{align*}
To bound the terms $A_{i}$, note first%
\[
A_{1}\leq\sup_{j\leq\tilde{N}}(1-d_{j})^{2}j^{-2\beta}M_{0,f}\leq N^{-2\beta
}M_{0,f}=\alpha^{-2\beta}n^{-2\beta/(2\beta+1)}(M+\gamma_{n})^{1/(2\beta+1)}.
\]
\begin{privatenotes}
{\bf Private Notes. } Note in the preceding display the term $(1 - d_{j})^{2} j^{-2\beta}$ is equal
to $N^{-2\beta}$ for $j\leq N$ and $j^{-2\beta}$ for $N < j \leq\tilde{N}$.
Therefore, the sup is $N^{-2\beta}$.
\end{privatenotes}Second, $A_{2}\leq\sum_{j>\tilde{N}}f_{j}^{2}\leq\tilde
{N}^{-2\beta}M=o(n^{-2\beta/(2\beta+1)})$. Furthermore,
\begin{align*}
A_{3}  &  =\frac{N}{n}\frac{1}{N}\sum(1-(j/N)^{\beta})_{+}^{2}\\
&  =\alpha n^{-2\beta/(2\beta+1)}M_{0,f}^{1/(2\beta+1)}\int_{0}^{\infty
}(1-t^{\beta})_{+}^{2}dt\;\;\;\left(  1+o(1)\right)  \text{ uniformly over
}f\in\Sigma\left(  \beta,M\right) \\
&  \leq\alpha n^{-2\beta/(2\beta+1)}M^{1/(2\beta+1)}\cdot\frac{2\beta^{2}%
}{(\beta+1)(2\beta+1)}\;\;\;\left(  1+o(1)\right)  .
\end{align*}
Combine these and choose $\alpha=\left(  \frac{(\beta+1)(2\beta+1)}{\beta
}\right)  ^{1/(2\beta+1)}$, and we find that the supremal risk, over
$f\in\Sigma\left(  \beta,M\right)  $, of the oracle estimator is at most
\begin{equation}
c(\beta)\cdot n^{-2\beta/(2\beta+1)}M^{1/(2\beta+1)}\;\;\;\left(
1+o(1)\right)  , \label{risk-oracle-estim}%
\end{equation}
where
\[
c(\beta)=\left(  \frac{\beta}{\beta+1}\right)  ^{2\beta/(2\beta+1)}%
\cdot(1+2\beta)^{1/(2\beta+1)}%
\]
is the Pinsker constant.

The next step is to show that the risk (\ref{risk-oracle-estim}) is also
attained when the unknown $M_{0,f}$ is replaced by an unbiased estimator. The
latter is $\hat{M}_{n}=\sum_{j=1}^{\tilde{N}_{n}}j^{2\beta}\hat{f}_{j}%
^{2}+\gamma_{n}$, where $\hat{f}_{j}^{2}=y_{j}^{2}-n^{-1}$. Then
\[
E(\hat{M})=\sum_{j=1}^{\tilde{N}_{n}}j^{2\beta}f_{j}^{2}+\gamma_{n}%
=M_{0,f}\leq M+\gamma_{n}%
\]
and
\begin{align*}
\text{Var}(\hat{M})  &  =\sum_{j=1}^{\tilde{N}}j^{4\beta}\text{Var}(Y_{j}%
^{2})\\
&  =\sum_{j=1}^{\tilde{N}}j^{4\beta}n^{-2}(2+4nf_{j}^{2})\\
&  =2n^{-2}\sum_{j=1}^{\tilde{N}}j^{4\beta}+4n^{-1}\sum_{j=1}^{\tilde{N}%
}j^{4\beta}f_{j}^{2}\\
&  =J_{1}+J_{2},
\end{align*}
where the first term
\[
J_{1}=2n^{-2}\tilde{N}^{4\beta+1}\cdot\frac{1}{\tilde{N}}\sum_{j=1}^{\tilde
{N}}\left(  j/\tilde{N}\right)  ^{4\beta}\sim2n^{-2}\tilde{N}^{4\beta+1}%
\cdot\int_{0}^{1}x^{4\beta}dx=o(1)
\]
since $\tilde{N}=o(n^{1/(2\beta+1/2)})$, and the second term
\[
J_{2}\leq4n^{-1}\tilde{N}^{2\beta}\sum_{j=1}^{\tilde{N}}j^{2\beta}f_{j}%
^{2}\leq4n^{-1}\tilde{N}^{2\beta}M=4M\;n^{-2}\tilde{N}^{4\beta+1}\frac
{n}{\tilde{N}^{2\beta+1}}=o(J_{1})
\]
uniformly for $f\in\Sigma(\beta,M)$ since $\tilde{N}\gg n^{1/(2\beta+1)}$.
Combining these gives $\mathrm{Var}(\hat{M})=o(1)$ uniformly for $f\in
\Sigma(\beta,M)$. Recalling $\gamma_{n}\gg\tilde{N}^{2\beta+1/2}/n$ gives
\[
\text{Var}\left(  \frac{\hat{M}-M_{0,f}}{\gamma_{n}}\right)  \sim
\frac{2Kn^{-2}\tilde{N}^{4\beta+1}}{\gamma_{n}^{2}}=o(1),
\]
and then
\[
\left\vert \frac{\hat{M}}{M_{0,f}}-1\right\vert \leq\left\vert \frac{\hat
{M}-M_{0,f}}{\gamma_{n}}\right\vert =o_{p}(1)
\]
uniformly.

Finally, it can be shown that the difference between the oracle estimator
$\left(  d_{j}Y_{j}\right)  _{1}^{\infty}$ and the estimator $\left(
d(j/N(\hat{M}))Y_{j}\right)  _{1}^{\infty}$ is negligible, i.e.
\[
E\sum_{j=1}^{\infty}\left(  d(j/N(M_{0,f}))-d(j/N(\hat{M}))\right)  ^{2}%
Y_{j}^{2}=o(n^{-2\beta/(2\beta+1)}).
\]

\begin{privatenotes}
{\bf Private Notes.} Optimal estimate: $\hat f_{j}^{2}= (1-\mu_{n} j^{\beta})_{+}$,
$\mu_{n}=\mu_{n}(M), \tilde\mu_{n} = \mu_{n}(\hat M)$,
\[
E\left( \hat M^{1/(2\beta+1)} - (E\hat M_{n} )^{1/(2\beta+1)} \right) ^{4} = ?
\]
\end{privatenotes}





\subsection{Proofs for Section \ref{sec:Bayes-minimax}%
\label{sec: proofs-Bayes-minimax}}

\begin{proof}
[Proof of Lemma \ref{lem-quad-tests-original}](a) Under the null hypothesis we
have $Y_{j}^{2}=n^{-1}\xi_{j}^{2}$, hence $T=\sum d_{j}\left(  \xi_{j}%
^{2}-1\right)  /\sqrt{2}$. Then it follows from (\ref{shrinking_coef}) and
$n\rho\rightarrow\infty$ that the CLT\ infinitesimality condition
\[
\sup_{j}d_{j}^{2}=o(1)
\]
holds uniformly over $d\in\mathcal{D}$, proving the assertion.

(b) Since $Y_{j}^{2}=f_{j}^{2}+2n^{-1/2}f_{j}\xi_{j}+n^{-1}\xi_{j}^{2}$, we
have
\begin{align}
T  &  =\frac{1}{\sqrt{2}}\sum d_{j}\left(  nf_{j}^{2}+2n^{1/2}f_{j}\xi
_{j}+\left(  \xi_{j}^{2}-1\right)  \right)  ,\label{in-view-of}\\
T-L(d,f)  &  =\frac{1}{\sqrt{2}}\sum d_{j}\left(  2n^{1/2}f_{j}\xi_{j}+\left(
\xi_{j}^{2}-1\right)  \right)  . \label{clt-infinitesi}%
\end{align}
An easy calculation gives
\[
\mathrm{Var}_{f}T=\frac{1}{2}\sum d_{j}^{2}\left(  4nf_{j}^{2}+2\right)
=1+2n\sum d_{j}^{2}f_{j}^{2}%
\]
where in view of (\ref{shrinking_coef}) we have for $f\in B_{\rho}^{\prime}$
\[
n\sum d_{j}^{2}f_{j}^{2}\leq\delta\rho^{-1}\sum f_{j}^{2}\leq2\delta=o(1).
\]
Consequently, $\mathrm{Var}_{f}T\rightarrow1$ uniformly. Now the CLT
infinitesimality condition on the sum (\ref{clt-infinitesi}) amounts to
\begin{equation}
\sup_{j}d_{j}^{2}\left(  nf_{j}^{2}+1\right)  =o(1). \label{clt-infinitesi-2}%
\end{equation}
For $f\in B_{\rho}^{\prime}$ we have $f_{j}^{2}\leq2\rho$, hence in view of
(\ref{shrinking_coef})
\[
d_{j}^{2}\left(  nf_{j}^{2}+1\right)  \leq d_{j}^{2}\left(  2n\rho+1\right)
\leq2\delta
\]
for $n$ sufficiently large. Hence (\ref{clt-infinitesi-2}) is fulfilled
uniformly over $d\in\mathcal{D}$ and $f\in B_{\rho}^{\prime}$, and the claim follows.

(c) Set $f_{j}\sim N(0,\sigma_{j}^{2})$; then in view of (\ref{in-view-of})%
\begin{equation}
T-L(d,\sigma)=\frac{1}{\sqrt{2}}\sum d_{j}\left(  2n^{1/2}f_{j}\xi_{j}+\left(
\xi_{j}^{2}-1\right)  \right)  +\frac{n}{\sqrt{2}}\sum d_{j}\left(  f_{j}%
^{2}-\sigma_{j}^{2}\right)  . \label{clt-infinitesi-3*}%
\end{equation}
An easy calculation gives
\begin{align*}
\mathrm{Var}_{f}T  &  =\frac{1}{2}\sum d_{j}^{2}\left(  4n\sigma_{j}%
^{2}+2\right)  +n\sum d_{j}^{2}\\
&  =1+n\sum d_{j}^{2}\left(  2\sigma_{j}^{2}+\sigma_{j}^{4}\right)
\end{align*}
where in view of (\ref{shrinking_coef}) we have for $\sigma\in B_{\rho
}^{\prime}$
\begin{align*}
n\sum d_{j}^{2}\sigma_{j}^{2}  &  \leq\delta\rho^{-1}\sum\sigma_{j}^{2}%
\leq2\delta=o(1),\\
n\sum d_{j}^{2}\sigma_{j}^{4}  &  \leq2\rho n\sum d_{j}^{2}\sigma_{j}^{2}%
\leq4\rho\delta=o(1).
\end{align*}
Consequently, $\mathrm{Var}_{f}T\rightarrow1$ uniformly. Now the
infinitesimality condition on the sum (\ref{clt-infinitesi-3*}) amounts to
\begin{equation}
\sup_{j}d_{j}^{2}\left(  1+n\sigma_{j}^{2}+n\sigma_{j}^{4}\right)  =o(1).
\label{clt-infinitesi-4}%
\end{equation}
For $\sigma\in B_{\rho}^{\prime}$ we have $\sigma_{j}^{2}\leq2\rho$, hence in
view of (\ref{shrinking_coef})
\[
d_{j}^{2}\left(  1+n\sigma_{j}^{2}+n\sigma_{j}^{4}\right)  \leq d_{j}%
^{2}\left(  1+n\rho+n\rho^{2}\right)  \leq3\delta
\]
for $n$ sufficiently large. Hence (\ref{clt-infinitesi-4}) is fulfilled
uniformly over $d\in\mathcal{D}$ and $\sigma\in B_{\rho}^{\prime}$, and the
claim follows.
\end{proof}

\bigskip

\begin{proof}
[Proof of Lemma \ref{lem-saddle} ]\bigskip Let $\mathcal{\tilde{D}}$ be
defined as $\mathcal{D}$ in (\ref{shrinking_coef}) but with condition
$\left\Vert d\right\Vert ^{2}=1$ replaced by $\left\Vert d\right\Vert ^{2}%
\leq1$. Then, since $L(d,f)$ is linear in $d$, for every $\tilde{d}%
\in\mathcal{\tilde{D}}$ there is a $d\in\mathcal{D}$ such that $L(\tilde
{d},f^{2})\leq L(d,f^{2})$ for every $f$. Hence it suffices to prove the claim
for $\mathcal{D}$ replaced by the compact convex set $\mathcal{\tilde{D}}$.
The restriction $f\in\Sigma(\beta,M)\cap B_{\rho}^{\prime}$ is equivalent to
$f^{2}$ being in the set
\begin{equation}
\left\{  g\in\mathbb{R}_{+}^{n}:\sum g_{j}j^{2\beta}\leq M,\rho\leq\sum
g_{j}\leq2\rho\right\}  \label{restric-in- terms-of-f2}%
\end{equation}
which is convex and compact (and nonempty for large enough $n$ since
$\rho\rightarrow0$). The functional $L$ is bilinear in $d$ and $f^{2}$; the
standard minimax theorem now furnishes the result.
\end{proof}

\bigskip

\begin{lemma}
\label{lem-descrip-saddlepoint}For $n$ large enough, the saddlepoint
$d_{0},f_{0}$ of Lemma \ref{lem-saddle} is given by
\[
d_{0}=\frac{f_{0}^{2}}{\left\Vert f_{0}^{2}\right\Vert },\;\;f_{0,j}%
^{2}=\left(  \lambda-\mu j^{2\beta}\right)  _{+}\text{, }j=1,\ldots,n
\]
where $\lambda,\mu$ are the unique positive solutions of the equations
\begin{equation}
\sum_{j=1}^{n}j^{2\beta}\left(  \lambda-\mu j^{2\beta}\right)  _{+}=M\text{,
}\sum_{j=1}^{n}\left(  \lambda-\mu j^{2\beta}\right)  _{+}=\rho.
\label{saddlepoint-shape}%
\end{equation}
The value of $L$ at the saddlepoint is
\begin{equation}
L_{0}=L(d_{0},f_{0})=\frac{n}{\sqrt{2}}\left\Vert f_{0}^{2}\right\Vert .
\label{saddlepoint-functional-value}%
\end{equation}

\end{lemma}

\begin{proof}
Ignore initially the restriction $\sup_{j}d_{j}^{2}\leq\delta/n\rho$ and
consider maximizing $L(d,f^{2})$ in $d$ for given $f$. Under the sole
restriction $\left\Vert d\right\Vert =1$, by Cauchy-Schwartz the solution is
found as
\[
d(f)=\frac{f^{2}}{\left\Vert f^{2}\right\Vert }.
\]
It remains to minimize $L(d(f),f)=n\left\Vert f^{2}\right\Vert /\sqrt{2}$
under the restrictions on $f^{2}.$ Setting $g_{j}=f_{j}^{2}$, one has to
minimize $\left\Vert g\right\Vert $ on the convex set
(\ref{restric-in- terms-of-f2}). This is solved using Lagrange multipliers
$\lambda,\mu$.

To show that the solution $d_{0}$ fulfills the restriction $\sup_{j}d_{j}%
^{2}\leq\delta/n\rho$, we note that
\begin{equation}
f_{0,j}^{2}=\left(  \lambda-\mu j^{2\beta}\right)  _{+}=\lambda\left(
1-\mu\lambda^{-1}j^{2\beta}\right)  _{+}\leq\lambda; \label{follows-1}%
\end{equation}
below (cf. (\ref{order-lambda}), Lemma \ref{lem-L0-value}) it is shown that
$\lambda\asymp n^{-1-1/(4\beta+1)}$ and $n\left\Vert f_{0}^{2}\right\Vert
\asymp L_{n,0}\asymp1$. This implies
\begin{align}
n\rho d_{0,n,j}^{2}  &  =n\rho\cdot O\left(  n^{2}\lambda^{2}\right)  \text{,
}\nonumber\\
n^{3}\rho\lambda^{2}  &  \asymp n\cdot n^{-4\beta/(4\beta+1)}\cdot
n^{-2/(4\beta+1)}=n^{-1/(4\beta+1)}; \label{follows-2}%
\end{align}
thus for $\delta=\left(  \log n\right)  ^{-1}$ we have that $d_{0}%
\in\mathcal{D}$ for $n$ large enough.
\end{proof}

\begin{proof}
[Proof of Lemma \ref{lem-Bayes-test}]The log-likelihood ratio is
\begin{align*}
&  \log\frac{\left(  n^{-1}\right)  ^{n/2}}{\left(  \sigma_{j}^{2}%
+n^{-1}\right)  ^{n/2}}\exp\left(  -\frac{1}{2}\sum_{j=1}^{n}\left(
\frac{Y_{j}^{2}}{\sigma_{j}^{2}+n^{-1}}-\frac{Y_{j}^{2}}{n^{-1}}\right)
\right) \\
&  =\frac{1}{2}\sum_{j=1}^{n}nY_{j}^{2}\left(  \frac{n\sigma_{j}^{2}}%
{n\sigma_{j}^{2}+1}\right)  -\frac{n}{2}\sum_{j=1}^{n}\log\left(  n\sigma
_{j}^{2}+1\right)  .
\end{align*}
This shows (a) by setting $d=\tilde{d}/\left\Vert \tilde{d}\right\Vert $ for
$\tilde{d}_{j}=\frac{n\sigma_{j}^{2}}{n\sigma_{j}^{2}+1}$. Now for $\sigma
_{j}^{2}=f_{0j}^{2}$ we have, as $\lambda\asymp n^{-1-1/(4\beta+1)}$,
\[
nf_{0j}^{2}=n\lambda\left(  1-\lambda^{-1}\mu j^{2\beta}\right)  _{+}\leq
n\lambda\asymp n\cdot n^{-1-1/(4\beta+1)}=n^{-1/(4\beta+1)}=o(1),
\]
hence $\tilde{d}_{j}\sim nf_{0j}^{2}$ uniformly over $j=1,\ldots,n$. This
implies $\left\Vert \tilde{d}\right\Vert \sim n\left\Vert f_{0}^{2}\right\Vert
\asymp n$ and
\[
d_{j}=\frac{\tilde{d}_{j}}{\left\Vert \tilde{d}\right\Vert }\asymp f_{0j}^{2}%
\]
uniformly in $j\leq n$. The proof of $n\rho d_{0,n,j}^{2}\leq\delta$ now
exactly follows (\ref{follows-1}), (\ref{follows-2}). The convergence
$t\rightarrow z_{\alpha}$ now is a consequence of Lemma
\ref{lem-quad-tests-original} (a).
\end{proof}

\begin{lemma}
\label{lem-L0-value}Suppose $\rho=c\cdot n^{-4\beta/(4\beta+1)}$, $c$
constant. Then the saddlepoint value $L_{0}$ of (\ref{value-game}) fulfills
\[
L_{0}=L(d_{0},f_{0}^{2})\sim\sqrt{A_{0}M^{-1/(2\beta)}c^{2+1/(2\beta)}/2}.
\]

\end{lemma}

\begin{proof}
The proof of Lemma \ref{lem-descrip-saddlepoint} shows that $L(d_{0},f_{0}%
^{2})$ is also the saddlepoint value under the weaker restrictions $\left\Vert
d\right\Vert ^{2}\leq1$, $f\in\Sigma(\beta,M)\cap B_{\rho}$. Let us sketch a
derivation of the asymptotics by a renormalization technique. Suppose that
$d_{j}=h^{1/2}d(hj)$, $j\leq n$ where $h$ is a bandwidth parameter tending to
$0$, and the continuous function $d:[0,\infty)\rightarrow\lbrack0,\infty)$
satisfies
\begin{equation}
\int_{0}^{\infty}d^{2}(x)\,dx\leq1. \label{constA}%
\end{equation}
Consider another continuous function $\sigma:[0,\infty)\rightarrow
\lbrack0,\infty)$ satisfying
\begin{equation}
\int_{0}^{\infty}x^{2\beta}\sigma^{2}(x)\,dx\leq1\ \text{ and }\ \int%
_{0}^{\infty}\sigma^{2}(x)\,dx\geq1 \label{constB}%
\end{equation}
and set $\sigma_{j}^{2}=Mh^{2\beta+1}\sigma^{2}(hj)$, $j\leq n$ . Choose
$h=(\rho/M)^{1/(2\beta)}$. The coefficient vector $d=\left(  d_{j}\right)
_{j=1}^{n}$ satisfies
\[
\left\Vert d\right\Vert ^{2}=h\sum_{j=1}^{n}d(hj)\rightarrow\int_{0}^{\infty
}d(x)dx\leq1.
\]
Identifying $f^{2}\in\mathbb{R}_{+}^{n}$ with $(\sigma_{j}^{2})_{j=1}^{n}$,
the restriction $f\in\Sigma(\beta,M)$ is asymptotically satisfied since
\[
\sum_{j=1}^{\infty}j^{2\beta}\sigma_{j}^{2}=Mh\sum_{j=1}^{\infty}(jh)^{2\beta
}\sigma^{2}(jh)\rightarrow M\int_{0}^{\infty}x^{2\beta}\sigma^{2}(x)\,dx\leq
M,\quad h\rightarrow0.
\]
The restriction $f\in B_{\rho}$ is also asymptotically satisfied since
\[
\sum_{j=1}^{\infty}\sigma_{j}^{2}=Mh^{2\beta+1}\sum_{j=1}^{\infty}\sigma
^{2}(jh)=\rho h\sum_{j=1}^{\infty}\sigma^{2}(jh)\sim\rho\int_{0}^{\infty
}\sigma^{2}(x)\,dx\geq\rho.
\]
Therefore,
\begin{align*}
\frac{n}{\sqrt{2}}\sum_{j=1}^{n}d_{j}\sigma_{j}^{2}  &  =\frac{n}{\sqrt{2}%
}Mh^{2\beta+1/2}h\sum_{j=1}^{\infty}d(jh)\sigma^{2}(jh)\\
&  \sim\frac{c^{1+1/(4\beta)}M^{-1/(4\beta)}}{\sqrt{2}}\int_{0}^{\infty
}d(x)\sigma^{2}(x)\,dx.
\end{align*}
The saddle point problem (\ref{value-game}) for each $n$ is thus
asymptotically expressed in terms of a fixed continuous problem with
constraints (\ref{constA}) and (\ref{constB}). There is unique positive
solution $(\lambda^{\ast},\mu^{\ast})$ for the equations (cp. \cite{Golub-82}%
),
\begin{align}
\int_{0}^{\infty}x^{2\beta}(\lambda-\mu x^{2\beta})\,dx  &  =1,\\
\int_{0}^{\infty}(\lambda-\mu x^{2\beta})\,dx  &  =1.
\end{align}
Let $\left\Vert \cdot\right\Vert _{2}$ and $\left\langle \cdot,\cdot
\right\rangle _{2}$ denote norm and scalar product in $L_{2}\left(
\mathbb{R}_{+}\right)  $. Then the saddle point $(d^{\ast},\sigma^{\ast2})$ is
given by
\begin{equation}
d^{\ast}=\frac{\sigma^{\ast2}}{\left\Vert \sigma^{\ast2}\right\Vert _{2}%
},\;\sigma^{\ast2}(x)=(\lambda^{\ast}-\mu^{\ast}x^{2\beta})_{+}.
\end{equation}
Then the value of the game is
\begin{align*}
\sup_{d\text{ in (\ref{constA})}}\,\inf_{\sigma\text{ in (\ref{constB})}%
}\left\langle d,\sigma^{2}\right\rangle _{2}  &  =\inf_{\sigma\text{ in
(\ref{constB})}}\sup_{d\text{ in (\ref{constA})}}\,\left\langle d,\sigma
^{2}\right\rangle _{2}\\
&  =\left\langle d^{\ast},\sigma^{\ast2}\right\rangle _{2}=\left\Vert
\sigma^{\ast2}\right\Vert _{2}=\sqrt{A_{0}(\beta)},
\end{align*}
where the sup is taken for $d$ satisfying (\ref{constA}), the inf is taken for
$\sigma$ satisfying (\ref{constB}), and $A_{0}(\beta)$ is Ermakov's constant
in (\ref{A-0-beta-def}). The continuous saddlepoint problem arises naturally
in a continuous Gaussian white noise setting and a parameter space described
by the continuous Fourier transformation, e.g. a Sobolev class of functions on
the whole real line (cf. \cite{Golub-82}, \cite{Golub-87}).

The above argument provides the guideline for a more rigourous proof, based on
calculating the sharp asymptotics of $\lambda$ and $\mu$ directly from
(\ref{saddlepoint-shape}). The rough order of $\lambda$ can be found as
follows. By equating $f_{0}^{2}=\sigma_{j}^{\ast2}$, we find
\begin{align*}
\left(  \lambda-\mu j^{2\beta}\right)  _{+}  &  =Mh^{2\beta+1}\sigma^{\ast
2}(hj),\\
&  =\lambda\left(  1-\left(  \left(  \mu/\lambda\right)  ^{1/2\beta}j\right)
^{2\beta}\right)  _{+}%
\end{align*}
we find $\lambda\asymp h^{2\beta+1}$ $,$ $h\asymp\left(  \mu/\lambda\right)
^{1/2\beta}$ and thus
\begin{equation}
\lambda\asymp h^{2\beta+1}\asymp\rho^{(2\beta+1)/(2\beta)}\asymp
n^{-1-1/(4\beta+1)}. \label{order-lambda}%
\end{equation}

\end{proof}

\begin{remark}
The paper of Ermakov \cite{Ermak-90}, when calculating the asymptotics of
$\lambda,\mu$ in (\ref{saddlepoint-shape}) and of $A=2L_{0}^{2}$ (in a more
general framework where $\sum a_{j}f_{j}^{2}\leq P_{0}$, $\sum b_{j}f_{j}%
^{2}\geq\rho$), contains an error for $\lambda$. Here is the correction using
the notations therein. Let $a_{j}=Lj^{2\gamma}$, $b_{j}=Mj^{2\nu}$, where
$\gamma>\nu\geq0$, $L$ and $M$ are positive constants, and set $\epsilon
=n^{-1/2}$. Then as $\epsilon\rightarrow0$ we have that
\[
\lambda\sim\frac{(2\gamma+2\nu+1)}{2(\gamma-\nu)}\left(  \frac{L}%
{P_{0}(4\gamma+1)}\right)  ^{\frac{4\nu+1}{2(\gamma-\nu)}}\left(  \frac{1}%
{M}\right)  ^{\frac{4\gamma+1}{2(\gamma-\nu)}}\left[  \rho(4\nu+1)\right]
^{\frac{2(\gamma+\nu)+1}{2(\gamma-\nu)}},
\]%
\[
\mu\sim\frac{(4\nu+1)\rho\lambda}{P_{0}(4\gamma+1)},\qquad A\sim\epsilon
^{-4}\rho\lambda\frac{4\gamma-4\nu}{4\gamma+1}.
\]

\end{remark}

\begin{tabular}
[c]{ll}%
\textsc{Department of Statistics\ \ \ \ \ \ \ \ \ \ } & \textsc{Department of
Mathematics}\\
\textsc{University of Georgia} & \textsc{Malott Hall}\\
\textsc{103 Statistics Building} & \textsc{Cornell University}\\
\textsc{101 Cedar St} & \textsc{Ithaca NY 14853}\\
\textsc{Athens GA 30605} & \textsc{e-mail:} nussbaum@math.cornell.edu\\
\textsc{e-mail:} psji@uga.edu\  &
\end{tabular}

\end{document}